\documentclass[arxiv,reqno,twoside,a4paper,12pt]{amsart}

\usepackage{amsmath, verbatim}
\usepackage{amssymb,amsfonts,mathrsfs,mathtools}
\usepackage[colorlinks=true,linkcolor=blue,urlcolor=blue,citecolor=blue]{hyperref}
\usepackage[driver=pdftex,heightrounded=true,centering]{geometry}
\usepackage{longtable, slashed}
\usepackage{tabularx}
\usepackage{multirow}
\usepackage{caption}
\usepackage{latexsym,enumitem}
\usepackage[all]{xy}
\usepackage[latin2]{inputenc}
\usepackage{t1enc, indentfirst}
\usepackage{graphicx, graphics}
\usepackage{pict2e, epic, amssymb }
\usepackage{tikz, pst-node}
\usepackage{tikz-cd, pgfplots}
\usetikzlibrary{arrows}
\usetikzlibrary{calc, patterns}
\pgfplotsset{compat=1.9}
\usepackage{epstopdf}
\usepackage{float}

\usepackage[toc]{appendix}
\usepackage[mathbf]{euler}

\usepackage[UKenglish]{babel}
\usepackage{tikz}
\usepackage{tikz-3dplot}

\usepackage{diagmac2} 
\usepackage{array} 
\usetikzlibrary{positioning}
\usetikzlibrary{calc}
\usetikzlibrary{through}
\usepackage{subcaption}
\allowdisplaybreaks 

\usepackage{xcolor}

\usepackage[scaled]{helvet} 
\usepackage{courier}

\theoremstyle{break}
\newtheorem{thm}{Theorem}[section]
\newtheorem{lem}[thm]{Lemma}%
\newtheorem{cor}[thm]{Corollary}
\newtheorem{defn}[thm]{Definition}
\newtheorem{rmk}[thm]{Remark} 

\DeclareMathOperator{\Id}{id}

\DeclareMathOperator{\divergence}{div}

\DeclareMathOperator{\Span}{span}
\DeclareMathOperator{\dvol}{dvol}
\DeclareMathOperator{\di}{d}

\DeclareMathOperator{\op}{op}
\DeclareMathOperator{\ff}{ff}
\DeclareMathOperator{\fd}{fd}
\DeclareMathOperator{\lf}{lf}
\DeclareMathOperator{\rf}{rf}
\DeclareMathOperator{\td}{td}
\DeclareMathOperator{\tb}{tb}
\DeclareMathOperator{\vol}{vol}

\numberwithin{equation}{section}

\definecolor{qqwuqq}{rgb}{0,0,0}

\begin{document}

\title[Schauder estimates on $\Phi$-manifolds]{Heat-type Equations on manifolds \\ with fibered boundaries I: \\ Schauder estimates}

\author{Bruno Caldeira}
\address{Universidade Federal de S\~{a}o Carlos, Brazil}
\email{brunoccarlotti@gmail.com}
\email{brunocarlotti@estudante.ufscar.br}

\author{Giuseppe Gentile}
\address{Universit\"{a}t Hannover, Germany}
\email{giuseppe.gentile@math.uni-hannover.de}

\subjclass[2020]{58J35; 35K05; 35K58}

\maketitle
\begin{abstract} 
In this paper we prove parabolic Schauder estimates for the Laplace-Beltrami operator on a 
manifold $M$ with fibered boundary and a $\Phi$-metric $g_\Phi$. This setting generalizes the 
asymptotically conical (scattering) spaces and includes special cases of magnetic and gravitational
monopoles.  
This paper, combined with part II, lay the crucial groundwork for forthcoming discussions on
geometric flows in this setting; especially the Yamabe- and mean curvature flow.
\end{abstract}

\tableofcontents

\section{Introduction and statement of the main results}\label{Introduction}

In this work we study the existence and regularity of solutions to certain partial differential equations on a specific class of manifolds with fibered boundary usually referred to as $\Phi$-manifolds.

\medskip
Here, by manifold with fibered boundary, we mean a compact manifold $\overline{M}$ with boundary $\partial \overline{M}$, where $\partial \overline{M}$ is the total space of a fibration $\phi:\partial \overline{M}\rightarrow Y$ over a closed (i.e. compact without boundary) Riemannian manifold $Y$.  
Furthermore, the fibers of the fibration $\phi$ are copies of a fixed closed manifold $Z$.
We say that the open interior $M$ of a manifold with fibered boundary $\overline{M}$ is a $\Phi$-manifold if it is equipped with what is known as a $\Phi$-metric.
Let a trivialization of the boundary of $\overline{M}$ be fixed, that is one chooses an identification with $\partial\overline{M}\times [0,\varepsilon)$.
A $\Phi$-metric $g_{\Phi}$ is a Riemannian metric on $M$ so that, when approaching the boundary $\partial\overline{M}$, it has asymptotic behavior described by 
\begin{equation}\label{FirstEquation}
g_{\Phi} = \dfrac{\di x^{2}}{x^{4}} + \dfrac{\phi^{*}g_{Y}}{x^{2}} + g_{Z} \; + h,
\end{equation}
where $h$ is the collection of cross-terms and it holds extra $x$ powers in each of its terms. In the above, $g_{Y}$ is a Riemannian metric on the base $Y$, while $g_{Z}$ is a symmetric bilinear form on $\partial \overline{M}$ which restricts to a Riemannian metric at each fiber. 
\medskip

The space $\mathbb{R}^m$ equipped with the Euclidean metric expressed in polar coordinates 
$$g=\di r^2+r^2\di \theta$$
is a noticeable example of a $\Phi$-manifold.
At a first glance the above metric $g$ may not resemble a metric of the form \eqref{FirstEquation}; also $\mathbb{R}^m$ is a well known non-compact manifold.
But, if one performs the change of coordinates $x=r^{-1}$ then the singular region $\{r=\infty\}$ "moves" to the origin $\{x=0\}$ resulting, in particular, in a boundary.
In particular the Euclidean metric expressed in the coordinates $(x,\theta)$ is of the form
$$g=\frac{\di x^2}{x^4}+\frac{\di \theta^2}{x^2}.$$
By comparing it with \eqref{FirstEquation} one notices that the term $g_Z$, as well as $h$, are vanishing. 
This leads to the conclusion that the metric can be described without the need of a $Z$ component, therefore implying that the resulting boundary is an example of a trivial fibration. 
Other noticeable example of $\Phi$-manifolds are: several complete Ricci-flat metrics, products of locally Euclidean spaces with a compact manifold and some classes of gravitational instantons.
\medskip

Although the concept of $\Phi$-manifolds may be traced back to works from the 1990's, they are a rather new field of investigation in geometric analysis, e.g. the analysis of geometric flows such as Yamabe-, Ricci-, mean curvature flow (just to name a few).
It is well known that the analysis of geometric flows relies on the analysis of certain parabolic partial differential equations. 
In particular those are of the form
\begin{equation}\label{baseq}
(\partial_{t} + a\Delta)u = \ell, \;\;u|_{t=0}=u_0,
\end{equation}
for some suitable functions $\ell$ and $a$ and $u_0$.  Thus here we focus our attention to the analysis of such PDE's.
\medskip

This work can be thought as a preparation for the analysis of the Yamabe- and the mean curvature flows on $\Phi$-manifolds, which will be presented in forthcoming works by the authors.  
This paper is the first of a two-parts work. 
It consist of a derivation of mapping properties for the heat-kernel operator $\mathbf{H}$.
Mapping properties of the heat-kernel will lead to solution of the differential equation \eqref{baseq} in the special case $a = 1$. 
The more general case, i.e. with $a$ being a function, will be presented by the same authors in part II of this work.

\subsection{Main results and structure of the paper}

In \S \ref{GeometrySection} we give a brief introduction to $\Phi$-manifolds along with their basic smooth structure, i.e. $\Phi$-vector fields, $\Phi$-differential operators.
Also, we introduce a suitable notion of H\"{o}lder regularity.  
In \S \ref{maxsection}, we prove $\Phi$-manifolds to be stochastically complete.
Stochastic completeness is a fact that will be extensively exploited in order to obtain the key result of this work, that is the mapping properties for the heat-kernel $\mathbf{H}$ presented below.
A quick overview of the heat space $\overline{M}^2_h$ (based on \cite{VerTal}) along with the various regimes and their projective coordinates can be found in \S \ref{HeatSpaceSec}.
In \S \ref{LiftsSection} we recall a result obtained by \cite{VerTal} describing the asymptotic behavior of the heat-kernel $\mathbf{H}$ on the various regimes.
\S \ref{MappingPropertiesSection} is devoted to presenting our first main result:

\begin{thm}\label{mappropH}
	Let $(M,g_{\Phi})$ be the open interior of a smooth compact manifold with fibered boundary endowed with a $\Phi$-metric $g_{\Phi}$ as in Definition \ref{phimetric}. Furthermore, consider any $\alpha \in (0,1)$, $T>0$ and any $\gamma \in \mathbb{R}$.  Then the heat-kernel operator $\mathbf{H}$ acts as a bounded map between weighted H\"{o}lder spaces
	\begin{align*}
	   &\mathbf{H}:x^{\gamma}C^{k,\alpha}_{\Phi}(M\times[0,T])\rightarrow x^{\gamma}C^{k+2,\alpha}_{\Phi}(M\times [0,T]), 
	\end{align*}
for every $k \in \mathbb{N}$, with the H\"{o}lder spaces $C^{k,\alpha}_{\Phi}(M\times [0,T])$ as defined below in \eqref{Holdernoweights}.
Moreover, a decrease in regularity on the target space yields additional weights as follows
\begin{align*}
	   &\mathbf{H}:x^{\gamma}C^{k,\alpha}_{\Phi}(M\times[0,T])\rightarrow \sqrt{t}x^{\gamma}C^{k+1,\alpha}_{\Phi}(M\times[0,T]), \\
	   &\mathbf{H}:x^{\gamma}C^{k,\alpha}_{\Phi}(M\times [0,T])\rightarrow t^{\alpha/2}x^{\gamma}C^{k+2}_{\Phi}(M\times [0,T]),
	\end{align*}
where $C^{k}_{\Phi}(M\times [0,T])$ is the space of functions whose $\Phi$-derivatives up to order $k$ are continuous on 
$\overline{M}\times [0,T]$.
\end{thm}
The proof for Theorem \ref{mappropH} will be carried over throughout \S \ref{HDSSection}, \S \ref{HDTSection} and \S \ref{SNSection}, where the computations for parabolic Schauder estimates are presented.
We want to point out that the methods employed in our proof are closely related to the one introduced by \cite{NEW} (in the setting of $b$-manifolds) consisting in the manifolds with corner description of the heat kernel to derive H\"{o}lder regularity through Schauder estimates (see also \cite[\S 3]{bahuaud2014yamabe}).  
Finally, in \S \ref{ShortTimeExistenceSec} will be presented a proof for the existence and regularity of solutions for a specific class of non-linear heat-type equations, 
that is tailored to the graphical mean curvature and Yamabe flows in this setting.

\begin{cor} \label{theorem4}
Let $\alpha\in (0,1)$ be fixed. 
Consider the Cauchy problem 
\begin{equation} \label{shorttime}
(\partial_{t} + \Delta)u = F(u), \;\;
u|_{t=0} = 0,
\end{equation}
Assume the map $F:x^{\gamma}C^{k+2,\alpha}_{\Phi}(M\times [0,T])\rightarrow C^{k,\alpha}_{\Phi}(M\times [0,T])$ to satisfy the following conditions: one can write $F = F_{1} + F_{2}$, with 
	\begin{enumerate}
	    \item $F_{1}:x^{\gamma}C^{k+2,\alpha}_{\Phi}\rightarrow x^{\gamma}C^{k+1,\alpha}_{\Phi}(M\times [0,T]),$
	    \item $F_{2}:x^{\gamma}C^{k+2,\alpha}_{\Phi}\rightarrow x^{\gamma}C^{k,\alpha}_{\Phi}(M\times [0,T])$
	\end{enumerate}
	and, for $u,u' \in x^{\gamma}C^{k+2,\alpha}_{\Phi}(M\times [0,T])$ satisfying $\|u\|_{k+2,\alpha,\gamma},\|u'\|_{k+2,\alpha,\gamma} \le \mu$, exists some $C_{\mu}>0$ such that
	
	\begin{enumerate}
		\item $\|F_{1}(u) - F_{1}(u')\|_{k+1,\alpha,\gamma} \le C_{\mu}\|u-u'\|_{k+2,\alpha,\gamma}$, $\|F_{1}(u)\|_{k+1,\alpha,\gamma} \le C_{\mu},$
		\item $\|F_{2}(u) - F_{2}(u')\|_{k,\alpha,\gamma} \le C_{\mu}\max\{\|u\|_{k+2,\alpha,\gamma},\|u'\|_{k+2,\alpha,\gamma}\}\|u-u'\|_{k+2,\alpha,\gamma}$, \newline $\|F_{2}(u)\|_{k,\alpha,\gamma} \le C_{\mu}\|u\|^{2}_{k+2,\alpha,\gamma}.$
	\end{enumerate}
	Then there exists a unique $u^{*} \in x^{\gamma}C^{k+2,\alpha}_{\Phi}(M\times [0,T'])$ solution for (\ref{shorttime}) for some $T'>0$ sufficiently small.
\end{cor}

\subsubsection*{Acknowledgements} The authors wish to thank Boris Vertman for the supervision as advisor for their Ph.D. theses.  The authors wish to thank the University of Oldenburg for the financial support and hospitality.  
The first author wishes also to thank the Coordena\c{c}\~{a}o de Aperfei\c{c}oamento de Pessoal de N\'{i}vel Superior (CAPES-Brasil- Finance Code 001) for the financial support (Process 88881.199666/2018-01).

\section{Geometry of fibered boundary manifolds}\label{GeometrySection}\label{PhiMfldsSect}

In this section, we briefly present some of the main concepts used throughout the whole paper. 
We refer to Mazzeo and Melrose \cite{mazzeo1998pseudodifferential}, as well as Talebi and Vertman \cite[\S 2]{VerTal}, for a more detailed and careful treatment of the subject.

\subsection{Manifolds with fibered boundary}\label{MFLDwithboundary}

The first step towards the definition of $\Phi$-manifolds is the definition of manifolds with fibered boundary.

\begin{defn}
Let $\overline{M}=M\cup \partial \overline{M}$ be a compact smooth manifold with boundary. 
We say that $\overline{M}$ is a manifold with fibered boundary if the boundary $\partial \overline{M}$ of $\overline{M}$ is the total space of a fibration, that is
\[\partial \overline{M} \xrightarrow{\phi} Y\]
with typical fiber $Z$ such that both $Y$ and $Z$ are closed manifolds
\end{defn}

Next we consider $(\partial\overline{M},g_{\partial\overline{M}})$ and $(Y,g_Y)$ to be Riemannian manifolds and the fibration $\phi$ to be a Riemannian submersion.
Recall that, $\phi$ is a Riemannian submersion if:
\begin{enumerate}
\item $\phi:\partial\overline{M}\rightarrow Y$ is surjective.
\item For every $p\in\partial\overline{M}$, $D_p\phi:T_p\partial\overline{M}\rightarrow T_{\phi(p)}Y$ is surjective.
\item For every $p\in\partial\overline{M}$, $D_p\phi:\ker(D_p\phi)^\perp\rightarrow T_{\phi(p)}Y$ is an isometry.
\end{enumerate}
Condition $(2)$ and $(3)$ allow for a "canonical" choice of a (Ehresmann) connection, i.e. a splitting of $T\partial\overline{M}$.
Let $T^V\partial\overline{M}=\ker(D\phi)$ be the vertical bundle.
An horizontal bundle $T^H\partial\overline{M}$ can be chosen to be the orthogonal complement of $T^V\partial\overline{M}$ with respect to $g_{\partial\overline{M}}$; thus leading to the splitting
$$T\partial\overline{M}=T^V\partial\overline{M}\oplus T^H\partial\overline{M}.$$
Also, condition $(3)$ implies 
$$g_{\partial\overline{M}}=\phi^*g_y+g_Z,$$
where $g_Z$ is a symmetric $(0,2)$-tensor on $\partial\overline{M}$ restricting to a positive definite tensor, i.e. a Riemannian metric, on each fiber.

We are now in the position to define a $\Phi$-metric.
Let $x$ be a choice of a boundary defining function for $\partial \overline{M}$, that is, $x\in C^{\infty}(\overline{M})$ is a non-negative function such that $\partial \overline{M} = \{p\in \overline{M}\,|\,x(p)=0\}$ and $Dx \neq 0$ on $\partial \overline{M}$.  
Since $\overline{M}$ is compact, there exists a collar neighborhood $U$ of $\partial \overline{M}$ in $\overline{M}$ such that $U \simeq [0,1) \times \partial \overline{M}$.  

\begin{defn}
A Riemannian metric $g_{\Phi}$ on the interior ${M}$ of $\overline{M}$ is called a $\Phi$-metric if, when restricted to the collar neighborhood $U$ above, it has an expression of the form
\begin{equation} \label{phimetric}
  g_{\Phi} = \dfrac{dx^{2}}{x^{4}} + \dfrac{\phi^{*}g_{Y}}{x^{2}} + g_{Z} + h =: \widehat{g} + h.
\end{equation}
As usual, $\widehat{g}$ is called the exact fibered boundary metric and $h$ is a perturbation (gathering all cross-terms in ${g_{\Phi}}$) such that $|h|_{\widehat{g}} = O(x)$ as $x\rightarrow 0$.
\end{defn}

From this point on, we will denote by $b$ the dimension of $Y$, by $f$ the dimension of $Z$, by $x$ a boundary defining function for $\partial\overline{M}$ and by $U$ a collar neighborhood of $\partial\overline{M}$, $U=[0,1)\times\partial\overline{M}$. 

\noindent
It is also important to point out the following.
Every point $p$ in $U$ can be parametrized by a pair $(x,w)$, with $x \in [0,1)$ and $w \in \partial \overline{M}$.  
Since $\partial \overline{M}$ is the total space of a fibration over the base space $Y$ with typical fiber $Z$, there is an open cover $\{V_{i}\}$ of $Y$ such that $\phi^{-1}(V_{i}) \simeq V_{i}\times Z$ for every $i$.
Thus, in such open subsets, every point can be written as a pair $(\widehat{y},\widehat{z})$.
It is, therefore, possible to write every point $p \in U$, locally, as the triple $p = (x,\widehat{y},\widehat{z})$. 
In conclusion, by means of the above identification, every point $p$ in $U$ has coordinates $p=(x,y^1,\dots,y^b,z^1,\dots,z^f)$, where $(y^1,\dots,y^b)$ and $(z^1,\dots,z^f)$ are coordinates for $\widehat{y}\in Y$ and $\widehat{z}\in Z$ respectively.
\begin{rmk}\label{CoordRemark}
Due to the abundance of indices, we will eventually use $y$ and $z$ to denote either the whole coordinate $(y^1,\dots,y^b)$ and $(z^1,\dots, z^f)$, respectively, or a generic coordinate element, i.e. $y=y^k$ for some $k$ and $z=z^j$ for some $j$.
\end{rmk}  
We want to conclude this section by noticing the following.
\begin{rmk}\label{OpenPhiMfld}
Although defined as compact manifolds with boundary, $\Phi$-manifolds model certain class of non-compact manifolds.
A prototypical example of this has already been discussed in \S \ref{Introduction}.
\end{rmk}

\subsection{$\Phi$-vector fields and $\Phi$-one forms} 

In the context of $\Phi$-manifolds, or in general of manifolds with boundary, one performs the analysis with respect to some specific set of vector fields. 
In particular, in the context of $\Phi$-manifolds, these ``well behaved'' vector fields, meaning that they take care of the singular nature of the metric tensor, are referred to as $\Phi$-vector fields and are defined as
\begin{align*}
    \mathcal{V}_{\Phi}(\overline{M}) = \left\{V \in \mathcal{V}(\overline{M})\,\bigg| \begin{array}{c}
         Vx \in x^{2}C^{\infty}(\overline{M}) \; \mbox{and} \\
         V_{p} \in T_{p}\phi^{-1}(\phi(p)) \; \mbox{for every} \; p \in \partial \overline{M}
    \end{array}\right\}.
\end{align*}
In local coordinates $(x,y,z)$ (cf. Remark \ref{CoordRemark})  near $\partial \overline{M}$,
\[\mathcal{V}_{\Phi}(\overline{M}) := \Span_{C^{\infty}(\overline{M})} \left\{x^{2}\partial_{x},x\partial_{y_{1}},...,x\partial_{y_{b}},\partial_{z_{1}},...,\partial_{z_{f}}\right\}.\]
The $\Phi$-tangent bundle $^{\Phi}T\overline{M}$ is, by definition, a vector bundle over $\overline{M}$ whose sections are given by $\mathcal{V}_{\Phi}(\overline{M})$.
\begin{rmk}
It is worth to point out that the inner product, i.e. the metric paring, of any two $\Phi$-vector fields is bounded.
\end{rmk}
Analogously to classical differential geometry, one can also consider the dual bundle $^{\Phi}T\overline{M}^{*}$, which is the bundle whose sections are differential forms on $\overline{M}$ generated by the family
\[\left\{\dfrac{\di x}{x^{2}},\dfrac{\di y_{1}}{x},...,\dfrac{\di y_{b}}{x},\di z_{1},...,\di z_{f}\right\}. \]

Once endowed with the family of $\Phi$-vector fields, it is reasonable to define the notion of $\Phi$-$k$-differentiability as
\begin{equation} \label{ckphi}
\begin{split}
&C^{1}_{\Phi}(\overline{M}) = \left\{u \in C^{0}(\overline{M})\; | \; Vu \in C^{0}(\overline{M}) \; \mbox{for every} \; V \in \mathcal{V}_{\Phi}(\overline{M})  \right\},\\
&C^{k}_{\Phi}(\overline{M}) = \left\{u \in C^{k-1}_{\Phi}(\overline{M})\; | \; Vu \in C^{k-1}_{\Phi}(\overline{M}) \; \mbox{for every} \; V \in \mathcal{V}_{\Phi}(\overline{M})\right\},
\end{split}
\end{equation}
where $k$ is an integer and $k\geq 2$.

Since $\mathcal{V}_{\Phi}(\overline{M})$ is a Lie algebra and a $C^{\infty}(\overline{M})$-module (see \cite{mazzeo1998pseudodifferential}), one can consider the algebra $\mbox{Diff}^{*}_{\Phi}(\overline{M})$ of higher order $\Phi$-differential operators.
This algebra consists of operators acting on $C^{\infty}(\overline{M})$ which can be written as a $C^\infty(\overline{M})$-linear combination of compositions of elements of $\mathcal{V}_{\Phi}(\overline{M})$.  
That is, we define the space $\mbox{Diff}^{k}_{\Phi}(\overline{M})$ of  $\Phi$-differential operators of order $k$, also denoted by $\mathcal{V}^{k}_{\Phi}$, as the space of linear operators $P:C^{\infty}(\overline{M})\rightarrow C^{\infty}(\overline{M})$ which can be locally expressed by
    \[P = \displaystyle\sum_{|\alpha| + |\beta| + q \le k} P_{\alpha,\beta,q}(x,y,z)(x^{2}\partial_{x})^{q}(x\partial_{y})^{\beta}\partial_{z}^{\alpha},\]
where $\alpha$ and $\beta$ are multi-indices, $\partial_{y} = \partial_{y_{1}},\dots,\partial_{y_{b}}$, $\partial_{z} = \partial_{z_{1}},\dots,\partial_{z_{f}}$ and $P_{\alpha,\beta,q}$ is a smooth function.

\subsection{H\"{o}lder continuity on $\Phi$-manifolds} \label{HoldPhiSect}

The analysis of PDEs relies upon the choice of suitable functional spaces. 
For our purposes, H\"{o}lder spaces and slight variations (weighted H\"{o}lder spaces) are needed.
Although we consider spaces similar to classical H\"{o}lder spaces, the H\"{o}lder spaces presented here encode the singular behavior of $\Phi$-metrics.
\medskip

Analogously to the spaces defined in \eqref{ckphi}, let us denote by $C^{k}_\Phi(M\times [0,T])$ the space of functions that, together with their $\Phi$-derivatives up to order $k$, are continuous on $\overline{M} \times [0,T]$. 
We want to point out, however, that time derivatives will be considered as second order derivatives.
For $\alpha \in (0,1)$, we define the $\alpha$-norm as the map $\|\cdot\|_{\alpha}:C^{0}(M\times [0,T])\rightarrow [0,\infty)$ given by
\begin{equation}\label{alphanorm}
    \|u\|_{\alpha}  = \|u\|_{\infty} + \mbox{sup}\left\{\dfrac{|u(p,t) - u(p',t')|}{d(p,p')^{\alpha} + |t-t'|^{\alpha/2}}\right\} =:  \|u\|_{\infty} + [u]_{\alpha}.
\end{equation}
The distance between $p$ and $p'$, appearing on the denominator of \eqref{alphanorm}, is defined in terms of $x^{4}g_{\Phi}$ and it is,  locally near the boundary, equivalent to
\begin{equation}\label{dist-func}
    d(p,p') = \sqrt{|x-x'|^{2}+(x+x')^{2}\|y-y'\|+(x+x')^{4}\|z-z'\|^{2}}.
\end{equation}
We define the $\alpha$-H\"{o}lder continuous functions as the space of functions continuous up to the boundary $\partial \overline{M}$ and whose $\alpha$-norm is bounded. 
That is
\[C^{\alpha}_{\Phi}(M\times [0,T]) := \{u \in C^{0}(\overline{M}\times [0,T])\,|\, \|u\|_{\alpha}<\infty\}.\]

Once endowed with the $\alpha$-norm \eqref{alphanorm}, this functional space turns into a Banach space. 
Higher order H\"{o}lder regularity is defined as follows.
For $k,l_{1}$ and $l_{2}$ being non-negative integers, the ($k,\alpha$)-H\"{o}lder space is given by
\begin{equation}\label{Holdernoweights}
    C^{k,\alpha}_{\Phi}(M\times [0,T]) = 
    \left\{
        u \in C^{k}_{\Phi}(M\times [0,T])\,\bigg|
        \begin{array}{l}  (\mathcal{V}^{l_{1}}_{\Phi}\circ \partial^{l_{2}}_{t})u \in C^{\alpha}_{\Phi}(M\times [0,T]),
           \\
        \mbox{for} \; l_{1} + 2l_{2} \le k 
        \end{array}\right\}
\end{equation}
From \cite[Proposition 3.1]{bahuaud2014yamabe} follows that the $(k,\alpha)$-H\"{o}lder space $C^{k,\alpha}_{\Phi}(M\times [0,T])$, when equipped with the norm 
\begin{equation*}
\|u\|_{k,\alpha} = \displaystyle\sum_{l_{1} + 2l_{2} \le k} \sum_{V \in \mathcal{V}^{l_{1}}_{\Phi}} \|(V\circ \partial_{t}^{l_{2}})u\|_{\alpha}
\end{equation*}
is also a Banach space.

\begin{rmk}
For every $0\le k_{1}\le k_{2}$ and for every $\alpha\in (0,1)$, one has 
$$C^{k_2,\alpha}_{\Phi}(M\times[0,T])\subset C^{k_1,\alpha}_{\Phi}(M\times[0,T]).$$
In particular, this means that, for every $k\ge 0$, $C^{k,\alpha}_{\Phi}(M\times[0,T])\subset C^{\alpha}_{\Phi}(M\times[0,T])$. 
\end{rmk}

Finally, for $\gamma$ a real number, one can define the weighted H\"{o}lder spaces as follow:
\begin{equation}
x^{\gamma}C^{k,\alpha}_{\Phi}(M\times [0,T]) = \{x^{\gamma}u\,|\, u \in C^{k,\alpha}_{\Phi}(M\times [0,T])\}
\end{equation}
On $x^{\gamma}C^{k,\alpha}_{\Phi}(M\times [0,T])$, consider then the modified norm \[\|x^{\gamma}u\|_{k,\alpha,\gamma}:= \|u\|_{k,\alpha}.\]
Whenever $\gamma\neq 0$, the above definition turns the multiplication by $x^{\gamma}$ into an isometry between $C^{k,\alpha}_{\Phi}(M\times [0,T])$ and $x^{\gamma}C^{k,\alpha}_{\Phi}(M\times [0,T])$, naturally implying that the weighted H\"{o}lder spaces are also Banach spaces.

\subsection{Classical H\"{o}lder spaces}

Due to the choice of the distance function $\di$ in \eqref{dist-func}, the H\"{o}lder spaces defined in \eqref{Holdernoweights} are quite "unnatural".
Here by "unnatural" we mean that the distance involved in the definition is not the distance induced by the $\Phi$-metric but rather by the conformal metric $x^4g_\Phi$; which is not representing the distance between two points in a $\Phi$-manifold.
The classical H\"{o}lder spaces are defined by taking $\di$ in \eqref{Holdernoweights} to be the distance induced by $g_\Phi$; that is
$$\di_\Phi(p,p')=\sqrt{\frac{|x-x'|^2}{(x+x')^4}+\frac{\|y-y'\|^2}{(x+x')^2}+\|z-z'\|^2}.$$
It is important to point out that both H\"{o}lder spaces are suitable functional spaces with a subtle difference.
In case of manifolds with boundary, the "unnatural" ones allow us to discuss continuity up to the boundary.
Thus leading to a better description of boundary behavior, a tool which will be useful in the analysis of geometric flows.
\medskip

In this work we derive mapping properties of the heat-kernel $\mathbf{H}$ in the unnatural H\"{o}lder spaces defined above.
It is important to point out that similar mapping properties to the one presented in \S \ref{MappingPropertiesSection} can be obtained, for classical H\"{o}lder spaces, by making use of parabolic Schauder and  Krylov-Safonov estimates, see e.g. \cite{krylov, KS, picard}. 

\noindent
Employing mapping properties of the parametrix for heat-type operators, allows for the analysis of various geometric flows.
The discussion above shows that, unnatural H\"{o}lder spaces are more suitable for the analysis of flows on manifolds with boundary.
Indeed one can deal with continuity up to the boundary, therefore providing more accurate bounds of geometric quantities at the boundary.
Classical H\"{o}lder spaces are instead more appropriate for gaining estimates in the interior.
Therefore they can be heavily employed in the analysis of geometric flows in open manifolds; see e.g. \cite{bruno, GeVe} for the analysis of the Yamabe- and the mean curvature flow in the setting of non-compact manifolds with bounded geometry.

\section{Stochastic Completeness} \label{maxsection}\label{StochPhiMflds}

As it will be proved later in this section, $\Phi$-manifolds, defined in \S\ref{PhiMfldsSect}, are an example of a much wider family of manifolds called stochastically complete manifolds.
\medskip

For convenience of the reader, we begin this section by setting up once and for all our convention for the Laplace-Beltrami operator, and we recall what a stochastically complete manifold is.

\begin{defn}
Let $(M,g)$ be an $m$-dimensional Riemannian manifold. 
For a function $u\in C^2({M})$, the Laplace-Beltrami operator is defined to be 
\begin{equation}\label{LaplacianConvention}
\Delta_{{g}} u=-\divergence\nabla u
\end{equation}
where $\divergence$ and $\nabla$ are the divergence and the gradient taken with respect to the metric tensor $g$ respectively.
In particular, given coordinates $(x^i)_{i=1,\dots,m}$, the local coordinate expression for the Laplacian is given by
\begin{equation}\label{lap-loc}
\Delta_{{g}} u=-\frac{1}{\sqrt{|g|}}\partial_i\left(\sqrt{|g|}{g}^{ij}\partial_j u\right).
\end{equation} 
In the above, $\partial_i$ is a short hand notation for $\partial/\partial x^i$, $|g|$ denotes the determinant of the metric tensor $g$ while ${g}^{ij}$ denotes its inverse.
It should also be noted that here and throughout the whole work, unless otherwise specified, we use the Einstein convention on repeated indices.
\end{defn}
\begin{defn}\label{StochCompleteDef} A Riemannian manifold $(M,g)$ is said to be stochastically complete if the heat kernel of the (positive) Laplace-Beltrami operator $\Delta$, associated to $g$ satisfies
\begin{equation}\label{StocComepleteIntegralDef}
\int_{{M}} H(t,p,\widetilde{p})\dvol_{g}(\widetilde{p})=1.
\end{equation}
\end{defn}
Stochastic completeness can be equivalently characterized by a volume growth condition, due to 
Grigor'yan \cite{grigor1986stochastically}, cf. also \cite[Theorem 2.11]{alias2016maximum}.

\begin{thm}\label{Gri}
Let $(M,g)$ be a complete Riemannian manifold. 
Consider for some reference point $p\in M$ the geodesic ball $B(p,R)$ of radius $R$ centered at $p$. 
If the function 
\begin{equation}\label{Grig'sFormula}
\overline{f}(\cdot) := \; \frac{\cdot}{\log\left(\vol\left(B(p,\cdot)\right)\right)}\not\in L^1(1,\infty)
\end{equation}
then $(M,g)$ is stochastically complete.
\end{thm} 

\medskip

In oder to check that $\Phi$-manifolds are stochastically complete, we begin by pointing out that $\Phi$-manifolds can be expressed as the union of a compact region $K$ with an open subset $U$, with $U$ equipped with the Riemannian metric locally given by the expression \eqref{phimetric}.
This is obtained by considering $U \simeq (0,1)\times \partial \overline{M}$ and identifying  $K = \{p\in \overline{M}\,|\,x(p) \ge 1\}$.

\medskip
Let $(M,g_{\Phi})$ be a $\Phi$-manifold.
By performing the change of coordinates $r = x^{-1}$ on $M$, one can rewrite the expression for $g_{\Phi}$ as
\begin{equation}
g_{\Phi} = \di r^{2} + r^{2}\phi^{*}g_{Y} + g_{Z} + h.
\end{equation}
With such a change of coordinates we note that, since both $Y$ and $Z$ are compact, the distance between two points towards the boundary ($\partial \overline{M} =\{r=\infty\}$) is proportional to $r$.
This can be checked by noticing that the distance from the boundary is given by the term $\di r^{2}$; therefore the distance in this direction is proportional to the Euclidean distance given in polar coordinates.

\medskip
In view of Theorem \ref{Gri}, let $p \in K$ (therefore away from the singular region) be fixed and consider $B(p,R)$ to be the open disc centred at $p$ of radius $R$.  
For any positive number $S$, let us consider the truncated compact subset $\overline{M}_{S} = \{q\in \overline{M}\,|\, r(q)\le S\}$.
As mentioned above, the distance on $(\overline{M},g_{\Phi})$ is proportional to $r$. 
Thus, there exists some fixed number $0<L<1/2$ such that, for $R>0$ large enough, the inclusion $B(p,R) \subset \overline{M}_{R/L}$ holds.
Therefore we conclude
\begin{equation*}
    \dfrac{R}{\log \vol B(p,R)} \sim \dfrac{R}{\log \vol \overline{M}_{R}}\;\; \mbox{as} \;\; R\rightarrow \infty,
\end{equation*}
meaning that the two functions agree up to bounded functions.
In particular, the latter is integrable if and only if the former is.
The expression for the $\Phi$-metric implies $\dvol_{\Phi}(p) = h_{0}r(p)^{b}\di r \di y \di z$, with $h_{0}$ being a bounded smooth function.
Hence, as $R$ goes to $\infty$, $\vol \overline{M}_{R} \sim R^{b+1} \le e^{CR^{2}}$, for some positive constant $C$.
That is
\begin{equation}
    \dfrac{\cdot}{\log \vol \overline{M}_{\cdot}} \notin L^{1}(1,\infty),
\end{equation}
implying, in particular, $(M,g_{\Phi})$ to be stochastically complete.

\section{Review of the heat space $M^{2}_{h}$}\label{HeatSpaceSec}

Section \ref{ShortTimeExistenceSec} will be devoted to proving the main results of this work; that is the existence of solutions for short time to heat-type Cauchy problems.
This will be achieved employing mapping properties of the heat-kernel operator $\mathbf{H}$. 

\medskip
Recall that the heat-kernel operator $\mathbf{H}$ is nothing but a convolution with the fundamental solution $H$ of the heat-equation; therefore turning $\mathbf{H}$ into a parametrix for the heat-operator $P=\partial_t+\Delta_{\Phi}$, with $\Delta_{\Phi}$ denoting the Laplace-Beltrami operator on a $\Phi$-manifold $(M,g_{\Phi})$.
As usual, we refer to such a function $H$ as the heat-kernel.
We want to point out that the heat-kernel can be seen as an element of a set of ``nice'' functions on $\overline{M}^{2}\times [0,\infty)_{t}$ (the polyhomogeneous functions).
For convenience to the reader, we will recall the definition of polyhomogeneous functions for a generic manifold with corners $M_C$.
We begin by defining an index family.

\medskip
A set $A \subset \mathbb{C}\times \mathbb{N}_{0}$ is called an index set if it satisfies the following conditions:
\begin{enumerate}
    \item $A$ is discrete and bounded from below; 
    \item $A_{j} := \{(\varsigma,p)\,| \, \mbox{Re}(\varsigma) < j\}$ is finite, for all $j$;
    \item If $(\varsigma,p) \in A$, then $(\varsigma + n,p) \in A$ for every $n \in \mathbb{N}$.
\end{enumerate}
A family $\mathcal{A} = (A_{1},...,A_{k})$ of index sets is called an index family.

\medskip
Next, let $M_C$ be a manifold with corners. 
Denote by $\{M_{C,1},...,M_{C,k}\}$ its family of boundary hypersurface and with $\left\{\rho_{M_{C,1}},...,\rho_{M_{C,k}}\right\}$ their respective boundary defining functions.
A function $H:M_C\rightarrow \mathbb{R}$ is polyhomogeneous with index family $\mathcal{E}$ if, near each boundary hypersurface $M_{C,j}$, $H$ admits the following asymptotic expansion:
\begin{equation} \label{phg}
    H\sim \displaystyle\sum_{(\varsigma,n)\in A_{j}}a_{j,\varsigma,n}\rho^{\varsigma}_{M_{C,j}}\left(\log\rho_{M_{C,j}}\right)^{n}, \; \mbox{as} \; \rho_{M_{C,j}}\rightarrow 0,
\end{equation}
with $a_{j,\varsigma,n}$ polyhomogeneous on $M_{C,j}$ with index family $(A_{1},...,\hat{A_{j}},,...,A_{k})$ near the intersections $M_{C,j}\cap M_{C,l}$ for any $l\neq j$.
In the above the $\hat{A_j}$ means that $A_j$ is not part of the index family.

\begin{rmk}
We would like to point out that the notation used above for boundary defining functions, i.e. $\rho_{\star}$ is the b.d.f. of $M_{\star}$, will be used across the entire work.
\end{rmk}

\medskip

In the setting of manifolds with fibered boundary, one has that the heat-kernel $H$ is a polyhomogeneous functions once considered over the heat space $M^2_h$.
Such a space $M^2_h$ is obtained from $\overline{M}^2\times[0,\infty)$ by replacing the regions where the heat-kernel $H$ is singular by their blow-ups. 

\medskip
The construction of the heat space is given by $3$ iterated blow-ups of $\overline{M}^{2}\times [0,\infty)_{t}$.  Such blow-ups are necessary to understand the asymptotic behavior of the heat kernel near its singular points. This can be done by replacing the singular regions by new boundary hypersurfaces. We refer to \cite{VerTal} for a more detailed discussion on both the construction of the heat space and the properties of the heat kernel given below.

\medskip
Before proceeding with the construction, we want to briefly recall the notion of p-submanifolds.
\begin{defn}
Let $\overline{M}$ denote an $m$-dimensional manifolds with boundary and consider $N\subset \overline{M}$ to be an $n$-dimensional submanifold. 
We say that $N$ is a p-submanifold if it can be locally expressed as $(x^{n+1}=0,\dots,x^{m}=0)$, where $(x^i)_i$ are local coordinates on $M$.
\end{defn}

\subsection{The first blow-up}

Consider first the submanifold $S_{1} = (\partial \overline{M})^{2}\times [0,\infty)_{t}$ of $\overline{M}^{2}\times [0,\infty)_{t}$.  
Notice that, since $\partial \overline{M}$ is a p-submanifold of $\overline{M}$, 
$S_1$ is a p-submanifold of $\overline{M}^2\times[0,\infty)$. 
By blowing up $S_1$ in $\overline{M}^2\times [0,\infty)_t$, we get the pair
$$M^{2}_{h,1}:= [\overline{M}^{2}\times [0,\infty)_{t};S_{1}],\;\;
\beta_{1}: M^{2}_{h,1} \rightarrow \overline{M}^{2}\times [0,\infty)_{t}.$$
The object $M^{2}_{h,1}$ is a "new" manifold obtained by cutting out the codimension $2$ submanifold of $\overline{M}^2\times [0,\infty)_t$ (displayed below as an edge) and gluing its spherical normal bundle (under appropriate identification) represented by a new boundary hypersurface (which is the conormal bundle of $S_{1}$ in $M^{2}\times [0,\infty)_{t}$).  
The new manifold $M^2_{h,1}$ comes equipped with a blowdown map $\beta_1: M^2_{h,1}\rightarrow\overline{M}^2\times[0,\infty)$.
The blowdown map is completely described by appropriate projective coordinates.
Before presenting the projective coordinates, we furnish the reader with a picture describing the blow up process from $\overline{M}^2\times[0,\infty)$ (right) to $M^2_{h,1}$ (left).
Some extra notation will be added into the picture, namely some reference to some specific faces.
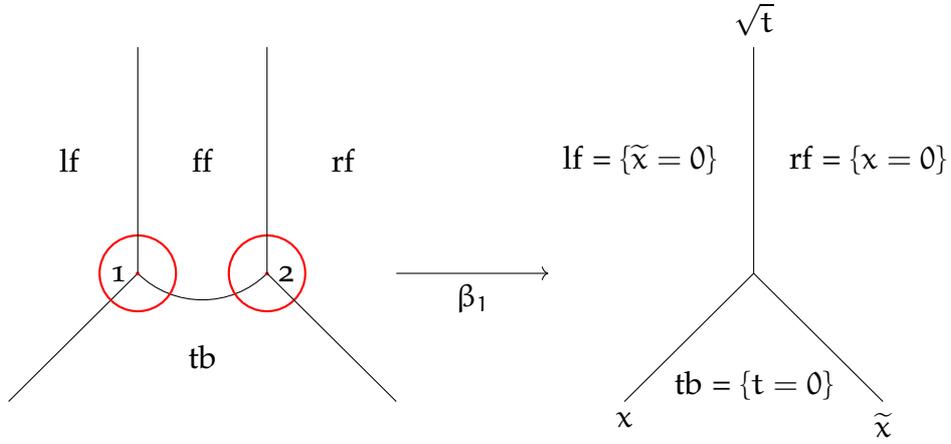
\begin{figure}[!htbp]
\centering
\begin{tikzpicture}
\draw (-8.8,-1.7)--(-7.1,0) node[left]{1};
\draw (-7.1,0) node[red,thick,circle,inner sep=9pt,draw]{.} to[out=315,in=225] (-5.4,0) node[red,thick,circle,inner sep=9pt,draw]{.};
\draw (-7.1,0) -- (-7.1,3);
\draw (-5.4,0) node[right]{2} to (-5.4,3);
\draw (-5.4,0) to (-3.7,-1.7);
\draw (-8,1.5) node{lf};
\draw (-6.25,1.5) node{ff};
\draw (-4.4,1.5) node{rf};
\draw (-6.25,-1.1) node{tb};
\draw[->] (-3.7,0)--(-1.7,0);
\draw (-2.7,0) node[below]{$\beta_{1}$};
\draw (1,0) to (-0.7,-1.7) node[left,below]{$x$};
\draw (1,0) to (1,3) node[above]{$\sqrt{t}$};
\draw (1,0) to (2.7,-1.7)  node[right,below]{$\widetilde{x}$};
\draw (-0.5,1.5) node{lf = $\{\widetilde{x}=0\}$};
\draw (2.5,1.5) node{rf = $\{x=0\}$};
\draw (1,-1.5) node{tb = $\{t=0\}$};
\end{tikzpicture}
\caption{First blow-up $M^{2}_{h,1}$} \label{firstblowup}
\end{figure}

Following the steps described in \cite{grieser2001basics}, one can describe the projective coordinates for $M^{2}_{h,1}$ by considering two regimes:

$\bullet$ \textbf{Regime near the intersection of lf, ff and tb:}  This regime is represented in the Figure \ref{firstblowup} by "regime 1". This regime is identified with the region where $\widetilde{x}\ll x$.
This implies, in particular, that the function $\widetilde{s} = \widetilde{x}^{-1}x$ is bounded.  Therefore, by writing $\sqrt{t} =: \tau$, the projective coordinates for the lower-left corner are
\begin{equation}\label{lowerleftcorner}
    \left(x,y,z,\dfrac{\widetilde{x}}{x},\widetilde{y},\widetilde{z},\sqrt{t}\right) = (x,y,z,\widetilde{s},\widetilde{y},\widetilde{z},\tau).
\end{equation}
Hence, on Regime 1 one has $\rho_{\ff} = x$, $\rho_{{\lf}} = \widetilde{s}$ and $\rho_{{\tb}} = \tau$, where we write $\rho_{\star}$ for a defining function of a boundary hypersurface $\star$.
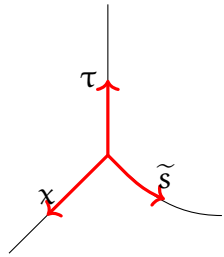
\begin{figure}[H]
    \centering
   \begin{tikzpicture}
\draw (0,0) to (-1.3,-1.3);
\draw (0,0) to (0,2);
\draw (0,0) to[out=315,in=180] (1.5,-0.8);
\draw[->,red,very thick] (0,0) to (-0.8,-0.8);
\draw[->,red,very thick] (0,0) to[out=315,in=157.5] (0.75,-0.585);
\draw[->,red,very thick] (0,0) to (0,1);
\draw (-0.8,-0.8) node[above]{$x$};
\draw (0.75,-0.585) node[above]{$\widetilde{s}$};
\draw (0,1) node[left]{$\tau$};
\end{tikzpicture}
    \caption{Regime near the intersection of lf, ff and tb}
    \label{lowerleftregime}
\end{figure}

$\bullet$ \textbf{Regime near the intersection of rf, ff and tb:}  This regime is represented in Figure \ref{firstblowup} by "regime 2", being identified with the case $x\ll \widetilde{x}$.  If $x \ll \widetilde{x}$ then $s = x^{-1}\widetilde{x}$ is a bounded function.  Hence, defining $\tau$ as above, the projective coordinates for the right-hand corner is
\begin{equation}\label{lowerrightcorner}
    \left(\sqrt{t},\dfrac{x}{\widetilde{x}},y,z,\widetilde{x},\widetilde{y},\widetilde{z}\right) = (\tau,s,y,z,\widetilde{x},\widetilde{y},\widetilde{z}).
\end{equation}
Similarly, on Regime 2 one has $\rho_{\ff} = \widetilde{x}$, $\rho_{\rf} = s$ and $\rho_{{\tb}} = \tau$.
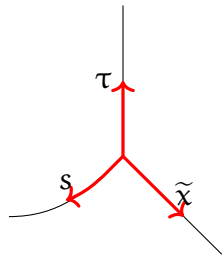
\begin{figure}[!htbp]
    \centering
    \begin{tikzpicture}
\draw (0,0) to (1.3,-1.3);
\draw (0,0) to (0,2);
\draw (0,0) to[out=225,in=0] (-1.5,-0.8);
\draw[->,red,very thick] (0,0) to (0.8,-0.8);
\draw[->,red,very thick] (0,0) to[out=225,in=22.5] (-0.75,-0.585);
\draw[->,red,very thick] (0,0) to (0,1);
\draw (0.8,-0.8) node[above]{$\widetilde{x}$};
\draw (-0.75,-0.585) node[above]{$s$};
\draw (0,1) node[left]{$\tau$};
\end{tikzpicture}
    \caption{Regime near the intersection of rf, ff and tb}
    \label{lowerrightregime}
\end{figure}

\begin{rmk}
The projective coordinates defined above for Regimes 1 and 2 are valid in "larger" regions.  
In fact, one can define both $s$ and $\widetilde{s}$ as long as one stays away from $\{\widetilde{x} = 0\}$ and $\{x = 0\}$ respectively.
This perspective will be useful for computing the parabolic Schauder estimates throughout \S \ref{HDSSection} to \S \ref{SNSection}.
\end{rmk}
We can finally give a precise expression for the blowdown map $\beta_1$.
We will focus only on regime $1$, since the local description of $\beta_{1}$ for regime $2$ follows similarly.
When restricted to the lower-left corner, the blowdown map takes the expression
\[(\beta_{1})\big|_{1}(\tau,x,y,z,\widetilde{s},\widetilde{y},\widetilde{z}) = (\tau,x,y,z,x\widetilde{s},\widetilde{y},\widetilde{z}).\]

\subsection{The second blow-up}
The second blow-up consists in blowing up the temporal fiber diagonal, meaning that we want to blow-up the submanifold $S_{2}$ of $M^{2}_{h,1}$ given by
\[S_2:=\left\{\dfrac{\widetilde{x}}{x} - 1 = 0 \hspace{2mm} \mbox{and} \hspace{2mm} y = \widetilde{y}\right\}.\]
To give a visual idea, the submanifold $S_2$ can be seen as a line in the middle of $\ff$ (in Figure \ref{firstblowup}) given by its intersection with the subspace $\{x = \widetilde{x}\}$.
As for the first blow-up, the "new" manifold can be pictured by replacing $S_{2}$ by its spherical inward pointing normal bundle (see Figure \ref{secondblowup}).  
The "new" manifold is defined by the pair 
$$M^{2}_{h,2} := [M^{2}_{h,1};S_{2}],\;\;\beta_{2}:M^{2}_{h,1}\rightarrow M^{2}_{h,1}$$
and it has a new boundary hypersurface given by fd = $\{\widetilde{s}-1=0 \hspace{2mm} \mbox{and} \hspace{2mm} y=\widetilde{y}\}$.
One can then consider the iterated blowdown map as the composition $\beta_{1}\circ \beta_{2}:M^{2}_{h,2}\rightarrow \overline{M}^{2}\times [0,\infty)_{\infty}$.
\begin{figure}[H]
    \centering
    \begin{tikzpicture}
\draw (0,0) node[left]{1} to (0,3);
\draw (0,0) to (-1.8,-1.8);
\draw (0,0) node[circle,inner sep=9pt,draw]{.} to[out=315,in=180] (1.2,-0.4);
\draw (1.2,-0.4) node[right]{3} to (1.2,3);
\draw (1.2,-0.4) node[red,thick,circle,inner sep=9pt,draw]{.} to[out=270,in=270] (3.1,-0.4);
\draw (3.1,-0.4) node[left]{4} to (3.1,3);
\draw (3.1,-0.4) node[red,thick,circle,inner sep=9pt,draw]{.} to[out=0,in=225] (4.3,0);
\draw (4.3,0) node[circle,inner sep=9pt,draw]{.} to (4.3,3);
\draw (4.3,0) node[right]{2} to (6.1,-1.8);
\draw (-1,1.5) node{lf};
\draw (0.6,1.5) node{ff};
\draw (2.15,1.5) node{fd};
\draw (3.7,1.5) node{ff};
\draw (5.3,1.5) node{rf};
\draw (2.15,-1.5) node{tb};
\end{tikzpicture}
    \caption{Second blow-up $M^{2}_{h,2}$}
    \label{secondblowup}
\end{figure}
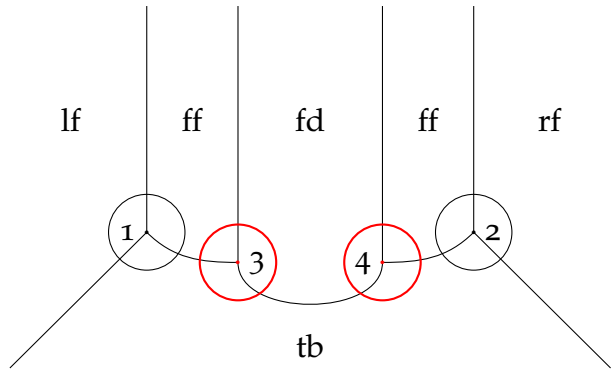
Following again the steps described in \cite{grieser2001basics}, it is possible to define the projective coordinates on fd by taking
\begin{equation}\label{coordinatesnearfd}
    \left(\tau,x,y,z,\dfrac{\widetilde{s}-1}{x},\dfrac{\widetilde{y}-y}{x},\widetilde{z}-z\right) =: \left(\tau,x,y,z,\mathcal{S}',\mathcal{U}',\mathcal{Z}'\right)
\end{equation}
away from $x = 0$ (which corresponds to "regime 3" in Figure \ref{secondblowup}).  Similarly, one can consider the projective coordinates on $\ff$ away from $\widetilde{x}=0$ (corresponding to "regime 4" in Figure \ref{secondblowup}) as
\[\left(\tau,\widetilde{x},\widetilde{y},\widetilde{z},\dfrac{s-1}{\widetilde{x}},\dfrac{y-\widetilde{y}}{\widetilde{x}},z-\widetilde{z}\right) =: \left(\tau,\widetilde{x},\widetilde{y},\widetilde{z},\widetilde{\mathcal{S}}', \widetilde{\mathcal{U}}',\widetilde{\mathcal{Z}}'\right).\]
\begin{rmk} \label{projr3}
Despite the projective coordinates given above for Regimes 3 and 4, one can actually use just one of the coordinates above to work on both Regimes, since one can understand that approaching $\ff$ from $\fd$ means that $\|(\mathcal{S}',\mathcal{U}',\mathcal{Z}')\|\rightarrow \infty$ (and similarly for $\|(\widetilde{\mathcal{S}}',\widetilde{\mathcal{U}}',\widetilde{\mathcal{Z}}')\|$).  Hence, one can say that on both Regimes 3 and 4, $\rho_{{\tb}} = \tau$, $\rho_{{\fd}} = x$ and one approaches $\ff$ if $\|(\mathcal{S}',\mathcal{U}',\mathcal{Z}')\|\rightarrow \infty$.
\end{rmk}
\begin{figure}[H]
    \centering
    \begin{tikzpicture}
\draw (0,0) to (0,3);
\draw (0,0) to[out=315,in=180] (1.2,-0.4);
\draw (1.2,-0.4) to (1.2,3);
\draw (1.2,-0.4) to[out=270,in=270] (3.1,-0.4);
\draw (3.1,-0.4) to (3.1,3);
\draw (3.1,-0.4) to[out=0,in=225] (4.3,0);
\draw (4.3,0) to (4.3,3);
\draw (0.6,1.5) node{ff};
\draw (2.15,1.5) node{fd};
\draw (3.7,1.5) node{ff};
\draw (0.7,-1.5) node{tb};
\draw[<->,red,very thick] (1.2,-0.4) to[out=270,in=270] (3.1,-0.4);
\draw[->,red,very thick] (2.15,-0.95) to (2.15,0.7);
\draw[->,red,very thick] (2.15,-0.95) to (2,-1.9);
\draw (2.4,-0.6) node[right]{$\mathcal{S}'$};
\draw (2.15,0) node[left]{$\tau$};
\draw (2,-1.6) node[right]{$x$};
\end{tikzpicture}
    \caption{Projective coordinates for the second blow-up}
    \label{projectivesecondblowup}
\end{figure}
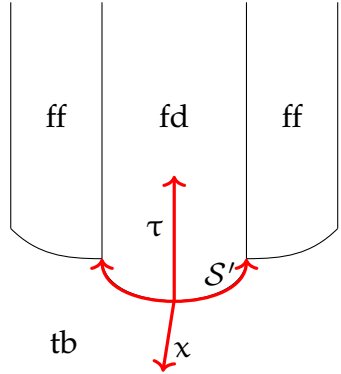
    
\subsection{The third blow-up}
We are now ready to present our third and last blow-up.
This blow-up arises from the classical singularity of the heat kernel on the spatial diagonal at time $t=0$. 
Therefore, the heat space $M^2_h$ is built by replacing $\mbox{diag}(M)\times \{t=0\}$ by its spherical normal bundle on $M^{2}_{h,2}$ (see Figure \ref{thirdblowup}).  
More precisely
$$\begin{array}{l}
M^{2}_{h} := \left[M^{2}_{h,2};(\beta_{1}\circ \beta_{2})^{-1}(\mbox{diag}(M)\times \{t=0\})\right], 
\; \beta: M^{2}_{h}\rightarrow \overline{M}^{2}\times [0,\infty)_{t}
\end{array}$$
with $\beta$ being the iterated blowdown map.  
Note that the heat space has one further boundary hypersurface td.
In particular this implies that $\mathcal{M}_{1}(M^{2}_{h}) = \{{\lf},\rf,{\tb},{\ff},{\fd},{\td}\}$ is the family of boundary hypersurfaces for $M^2_h$.
\begin{figure}[H]
    \centering
    \begin{tikzpicture}
\draw (0,0) node[left]{1} to (0,3);
\draw (0,0) to (-1.8,-1.8);
\draw (0,0) node[circle,inner sep=7pt,draw]{.} to[out=315,in=180] (1.2,-0.4);
\draw (1.2,-0.4) node[right]{3} to (1.2,3);
\draw (1.2,-0.4) node[circle,inner sep=7pt,draw]{.} to[out=270,in=180] (1.6,-0.8);
\draw (1.6,-0.8) to (1.6,-1.8);
\draw (1.6,-0.8) to[out=90,in=90] (2.7,-0.8);
\draw (2.7,-0.8) to (2.7,-1.8);
\draw (2.7,-0.8) to[out=0,in=270] (3.1,-0.4);
\draw (3.1,-0.4) node[left]{4} to (3.1,3);
\draw (3.1,-0.4) node[circle,inner sep=7pt,draw]{.} to[out=0,in=225] (4.3,0);
\draw (4.3,0) node[circle,inner sep=7pt,draw]{.} to (4.3,3);
\draw (4.3,0) node[right]{2} to (6.1,-1.8);
\draw (2.15,-0.495) node[red,thick,circle,inner sep=9pt,draw]{.};
\draw (2.15,-0.495) node[above]{5};
\draw (-1,1.5) node{lf};
\draw (0.6,1.5) node{ff};
\draw (2.15,1.5) node{fd};
\draw (3.7,1.5) node{ff};
\draw (5.3,1.5) node{rf};
\draw (0.6,-1.5) node{tb};
\draw (2.15,-1.5) node{td};
\draw (3.7,-1.5) node{tb};
\end{tikzpicture}
    \caption{Third blow-up}
    \label{thirdblowup}
\end{figure}
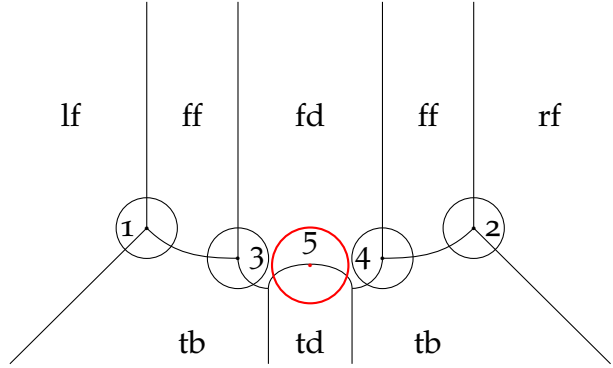
The projective coordinates near the intersection of fd and td (which is represented by "regime 5" in Figure \ref{thirdblowup}) are given by
\begin{equation}\label{coordiantesneartd}
    \left(\tau,x,y,z,\dfrac{\mathcal{S}'}{\tau},\dfrac{\mathcal{U}'}{\tau},\dfrac{\mathcal{Z}'}{\tau}\right) =: (\tau,x,y,z,\mathcal{S},\mathcal{U},\mathcal{Z}).
\end{equation}
Using the same notations as in the previous blow-ups, we have the boundary defining functions $\rho_{{\fd}} = x$, $\rho_{{\td}} = \tau$ and $\|(\mathcal{S},\mathcal{U},\mathcal{Z})\|\rightarrow \infty$ corresponds to $\tb$.
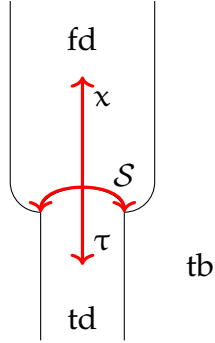
\begin{center}
\begin{tikzpicture}
\draw (1.2,-0.4) to (1.2,2);
\draw (1.2,-0.4) to[out=270,in=180] (1.6,-0.8);
\draw (1.6,-0.8) to (1.6,-2.5);
\draw (1.6,-0.8) to[out=90,in=90] (2.7,-0.8);
\draw (2.7,-0.8) to (2.7,-2.5);
\draw (2.7,-0.8) to[out=0,in=270] (3.1,-0.4);
\draw (3.1,-0.4) to (3.1,2);
\draw (2.15,1.5) node{fd};
\draw (2.15,-2.2) node{td};
\draw (3.7,-1.5) node{tb};
\draw[<->,red,very thick] (1.6,-0.8) to[out=90,in=90] (2.7,-0.8);
\draw[<->,red,very thick] (2.15,-1.5) to (2.15,1);
\draw (2.15,-1.2) node[right]{$\tau$};
\draw (2.15,0.7) node[right]{$x$};
\draw (2.7,-0.6) node[above]{$\mathcal{S}$};
\end{tikzpicture}
\end{center}
\begin{rmk}
In the interior of $\td$, away from $\fd$, we can also use projective coordinates $(\tau,(\theta-\widetilde{\theta})/\tau,\widetilde{\theta})$, where $\theta,\widetilde{\theta}$ are two copies of any local coordinates on $M$.
\end{rmk}

\section{Asymptotic expansion of the heat kernel on $M^2_h$}\label{LiftsSection}

In the previous section, we gave a quick overview of the heat space $M^{2}_{h}$ and presented the main regimes we will be interested in.
Recall that $(M,g_{\Phi})$ will always denote a $\Phi$-manifold and $H$ the heat kernel of the Laplace-Beltrami operator $\Delta_{\Phi}$ associated to the $\Phi$-metric $g_{\Phi}$.

\medskip
The aim of this section is to recall a result, obtained by \cite{VerTal}, describing the asymptotic behavior of $\beta^{*}H$ on $M^{2}_{h}$ when approaching each of its boundary hypersurfaces in $\mathcal{M}_1(M^2_h)$ described in \S \ref{HeatSpaceSec}.
\medskip


\begin{thm}\cite[Theorem 7.2]{VerTal}\label{MohammadThm}
Let $(M,g_{\Phi})$ be an $m$-dimensional complete manifold with fibered boundary endowed with a $\Phi$-metric. Denote by $H$ the heat kernel associated to the unique self-adjoint extension of the corresponding Laplace-Beltrami operator. The lift $\beta^*H$ is a polyhomogeneous function on $M^2_h$ with asymptotic behavior described by 
\begin{equation}
\beta^*H\sim \rho_{\lf}^\infty \rho_{\ff}^\infty\rho_{\rf}^\infty\rho_{\tb}^\infty\rho_{\fd}^0\rho_{\td}^{-m}G_{0}
\end{equation}
with $G_{0}$ being a bounded function.
In particular, the above means that $\beta^{*}H$ is of leading order $-m$ on $\td$, smooth and bounded on $\fd$ and vanishes to infinite order on $\lf$, $\ff$, $\rf$ and $\td$.  
\end{thm} 
We can now describe the asymptotic behaviors of objects, which will be of particular interest in the next sections, neat the regimes introduced in \S \ref{HeatSpaceSec}.

\begin{rmk}
In the following formulas we will be describing only the worst case scenario, i.e. the most singular behavior. Bounds for better behaved terms in the asymptotic expansion follow from the most singular one and the error estimate from the definition of polyhomogeneous conormal functions. 
\end{rmk}

\textbullet \textbf{Lifts in the intersection of lf, ff and tb:} In this regime, the asymptotic behaviors of both $\beta^*H$ and $\beta^{*}\dvol_{\Phi}\di \widetilde{t}$ are appropriately described by the projective coordinates
\begin{align*}
x,\;\;y,\;\;z,\;\;\widetilde{s}=\frac{\widetilde{x}}{x},\;\;\widetilde{y},\;\;\widetilde{z},\;\;\tau=\sqrt{t}.
\end{align*}
Recall that, with respect to these coordinates, $\rho_{\lf}=\widetilde{s}$, $\rho_{\ff}=x$ and $\rho_{\tb}=\tau$. From Theorem \ref{MohammadThm} and by computing directly the pull-back of the volume form, one has
\begin{equation}\label{LiftNear_lf}
\begin{split}
&\beta^*(XH)\sim\tau^{-1}(x\widetilde{s}\tau)^\infty G_{0}=(x\widetilde{s}\tau)^\infty G_{0}\;\;\mbox{ for }X\in\{\Id, \mathcal{V}_\phi,\mathcal{V}^2_\phi,\partial_t\}\\
&\beta^*(\dvol_\Phi \di\widetilde{t})\sim2(\widetilde{s}x)^{-2-b}\tau x h\di \widetilde{s}\di\widetilde{y}\di\widetilde{z}\di\tau.
\end{split}
\end{equation}

\textbullet\textbf{Lifts in the intersection of rf, ff and tb:} The asymptotic behaviors of both $\beta^*H$ and $\beta^{*}\dvol_{\Phi}\di \widetilde{t}$ are suitably described by the projective coordinates
$$s=\frac{x}{\widetilde{x}},\;\;y,\;\;z,\;\;\widetilde{x},\;\;\widetilde{z},\;\;\tau=\sqrt{t}.$$
The boundary defining function with respect to these coordinates are $\rho_{\rf}=s$, $\rho_{\ff}=\widetilde{x}$ and $\rho_{\tb}=\tau$. As above we can conclude that 
\begin{equation}\label{LiftNear_rf}
\begin{split}
&\beta^*(XH)\sim\tau^{-1}(\widetilde{x}s\tau)^\infty G_{0}=(\widetilde{x}s\tau)^\infty G_{0}\;\;\mbox{ for } X\in\{\Id,\mathcal{V}_\phi,\mathcal{V}^2_\phi\}\\
&\beta^*(\dvol_\Phi\di\widetilde{t})\sim2\widetilde{x}^{-2-b}\tau h\di\widetilde{x}\di\widetilde{y}\di\widetilde{z}\di\tau.
\end{split}
\end{equation}

\textbullet\textbf{Lift in the intersection of ff, fd and tb:} In this regime, the asymptotic behaviors of both $\beta^*H$ and $\beta^{*}\dvol_{\Phi}\di \widetilde{t}$ are fittingly described using projective coordinates
$$x,\;\;y,\;\;z,\;\;\mathcal{S}'=\frac{\widetilde{x}-x}{x^2},\;\;\mathcal{U}'=\frac{y-\widetilde{y}}{x},\;\;\mathcal{Z}'=z-\widetilde{z},\;\;\tau=\sqrt{t}.$$
In the next section we will encounter some extra quantity in this regime hence it is useful to collect it here. As in the previous cases, one has
\begin{equation}\label{SecondBlowUpProjCoord1}
\begin{split}
&\beta^*(XH)=\tau^{-1}(\tau)^\infty G_{0}\sim\tau^\infty G_{0}\;\; \mbox{ for }X\in\{\Id,\mathcal{V}_\phi,\mathcal{V}^2_\phi\}.\\
&\beta^*(\dvol_\Phi\di\widetilde{t})\sim2(1+\mathcal{S}'x)^{-2-b}\tau h\di\mathcal{S}'\di\mathcal{U}'\di\mathcal{Z}'\di\tau.\\
&\beta^*(\partial_i XH)\sim x^{-2}\tau^\infty G_{0}\;\;\mbox{ with }i=x,y,z.
\end{split}
\end{equation}
Note that, in the above, $G_{0}$ vanishes to infinite order as $\|(\mathcal{S}',\mathcal{U}',\mathcal{Z}')\|$ goes to $\infty$.

\textbullet\textbf{Lifts in the intersection of fd and td:}
The appropriate projective coordinate for a suitable description of the asymptotic behavior of both $\beta^*H$ and $\beta^{*}\dvol_{\Phi}\di \tilde{t}$ in this regime are
$$x,\;\;y,\;\;z,\;\;\mathcal{S}=\frac{\widetilde{x}-x}{\tau x^2},\;\;\mathcal{U}=\frac{y-\widetilde{y}}{\tau x},\;\;\mathcal{Z}=\frac{z-\widetilde{z}}{\tau},\;\;\tau=\sqrt{t}.$$
Similarly to the regime where $\ff$ intersects $\fd$ and $\tb$ one finds
\begin{equation}\label{ABHK3}
\begin{split}
&\beta^*(XH)\sim(\tau)^{-m-2} G_{0}\;\;\mbox{ for }X\in\{\Id,\mathcal{V}_\phi,\mathcal{V}^2_\phi\}.\\
&\beta^*(\dvol_\Phi\di\widetilde{t})\sim2(1+\mathcal{S}\tau x)^{-2-b}\tau^{m+1} h\di\mathcal{S}\di\mathcal{U}\di\mathcal{Z}\di\tau.\\
&\beta^*(\partial_i XH)\sim x^{-2}\tau^{-m-3} G_{0}\;\;\mbox{ with }i=x,y,z.
\end{split}
\end{equation}
with $G_{0}$ vanishing to infinite order as $\|(\mathcal{S},\mathcal{U},\mathcal{Z})\|$ goes to $\infty$.

\section{Mapping properties of the heat-kernel on $M^2_h$}\label{MappingPropertiesSection}

Our aim is to prove the existence of solution of heat-type equations on manifolds with fibered boundary equipped with a $\Phi$-metric.
That is, we focus on the analysis of a Cauchy problem of the form
\begin{equation}
(\partial_t+\Delta)u=\ell,\;\;
u|_{t=0}=u_0
\end{equation}
for suitable functions $\ell$ and $u_0$. The proof of the existence of solutions for the above problem relies on the analysis of the mapping properties of the heat operator acting via time convolution between H\"{o}lder spaces. We will provide a slightly more general mapping properties of the heat operator between weighted H\"{o}lder spaces.
\begin{thm}\label{MappingPropertiesTHM}
Let $M$ be an $m$-dimensional manifold with fibered boundary equipped with a $\Phi$-metric. 
The heat operator $H$, acting by convolution, 
\begin{equation}\label{ClaimThm5.1}
\mathbf{H}:x^\gamma C^{k,\alpha}_\Phi(M\times [0,T])\rightarrow x^\gamma C^{k+2,\alpha}_\Phi(M\times[0,T])
\end{equation}
is bounded.
\end{thm}
\begin{proof}
We will prove the statement for $k=0$, that is 
$$\mathbf{H}:x^\gamma C^{\alpha}_\Phi(M\times [0,T])\rightarrow x^\gamma C^{2,\alpha}_\Phi(M\times[0,T]).$$
The more general case can be proved similarly with additional integration by part argument near $\td$ and by employing the vanishing order of the heat kernel near the boundary $\partial M^{2}_{h}$ (similar argument has been employed in \cite{bahuaud2014yamabe}). Furthermore, note that, for $\overline{u}\in x^\gamma C^\alpha_\Phi(M\times[0,T])$, there exists some $u\in C^\alpha_\Phi(M\times [0,T])$ so that $\overline{u}=x^\gamma u$. In particular, it follows that $H\overline{u}$ lies in $x^\gamma C^{2,\alpha}_\Phi(M\times[0,T])$ if and only if $x^{-\gamma} Hx^\gamma u$ lies in $C^{2,\alpha}_\Phi(M\times [0,T])$.
This is equivalent to prove that 
\begin{equation}\label{Hgamma}
\mathbf{H}_{\gamma} := M(x^{-\gamma})\circ \mathbf{H}\circ M(x^{\gamma}):C^\alpha_\Phi(M\times[0,T])\rightarrow C^{2,\alpha}_\Phi(M\times[0,T]),
\end{equation}
is bounded, with $M(x^{\gamma})$ being the ''multiplication by $x^{\gamma}$'' operator.  
Moreover, from (\ref{Holdernoweights}), it follows that proving \eqref{ClaimThm5.1} is equivalent to prove that the operator $\mathbf{G}$, defined by $\mathbf{G}=V\mathbf{H}_{\gamma}$ with $V\in\{\Id,\mathcal{V}_\Phi,\mathcal{V}^2_\Phi\}$, is a bounded operator mapping
$$\mathbf{G}:C^\alpha_\Phi(M\times[0,T])\rightarrow C^\alpha_\Phi(M\times [0,T]).$$
Therefore, given a function $u$ in $C^\alpha_\Phi(M\times[0,T])$, the goal is to prove
\begin{equation}\label{AimedInequality}
\|\mathbf{G}u\|_\alpha\le c\|u\|_\alpha
\end{equation}
for some uniform constant $c>0$. This will be obtained directly by estimating $\|\mathbf{G}u\|_\alpha$. From the definition of the $\alpha$-norm in (\ref{alphanorm}) we find
$$\|\mathbf{G}u\|_\alpha=[\mathbf{G}u]_\alpha+\|\mathbf{G}u\|_\infty.$$
One can see that 
$$[\mathbf{G}u]_\alpha\le \sup_{\substack{p,p'\in M\\p\neq p'}}\frac{|\mathbf{G}u(p,t)-\mathbf{G}u(p',t)|}{\di (p,p')^\alpha}+\sup_{\substack{t,t'\ge 0\\t\neq t'}}\frac{|\mathbf{G}u(p,t)-\mathbf{G}u(p,t')|}{|t-t'|^{\alpha/2}},$$
leading to
\begin{align*}
\|\mathbf{G}u\|_\alpha\le& \sup_{\substack{p,p'\in M\\p\neq p'}}\frac{|\mathbf{G}u(p,t)-\mathbf{G}u(p',t)|}{\di (p,p')^\alpha}+\sup_{\substack{t,t'\ge 0\\t\neq t'}}\frac{|\mathbf{G}u(p,t)-\mathbf{G}u(p,t')|}{|t-t'|^{\alpha/2}}+\|\mathbf{G}u\|_\infty.
\end{align*}
Thus \eqref{AimedInequality} is satisfied if the following are satisfied
\begin{align}
&|\mathbf{G} u(p,t)-\mathbf{G}u(p',t)|\le c\|u\|_\alpha d(p,p')^\alpha, \label{HDS}\\
&|\mathbf{G} u(p,t)-\mathbf{G}u(p,t')|\le c\|u\|_\alpha |t-t'|^{\alpha/2}, \label{HDT}\\
&|\mathbf{G}u(p,t)|\le c\|u\|_\alpha. \label{SN}
\end{align}
We will therefore proceed in three steps:
\begin{itemize}
\item[i)] Uniform estimates of H\"{o}lder differences in space \eqref{HDS},
\item[ii)]Uniform estimates of H\"{o}lder differences in time \eqref{HDT},
\item[iii)]Uniform estimates of the supremum norm \eqref{SN}.
\end{itemize}
These three steps will be treated separately in sections \ref{HDSSection}, \ref{HDTSection} and \ref{SNSection} respectively.
\end{proof}
From Theorem \ref{MappingPropertiesTHM} other mapping properties can be derived.
\begin{thm}\label{SecondMappingPropertyTHM}
Let $M$ be an $m$ dimensional manifold with fibered boundary equipped with a $\Phi$-metric. The following operators acting by convolution in time
\begin{align*}
	   &\mathbf{H}:x^{\gamma}C^{k,\alpha}_{\Phi}(M\times[0,T])\rightarrow \sqrt{t}x^{\gamma}C^{k+1,\alpha}_{\Phi}(M\times[0,T]), \\
	   &\mathbf{H}:x^{\gamma}C^{k,\alpha}_{\Phi}(M\times [0,T])\rightarrow t^{\alpha/2}x^{\gamma}C^{k+2}_{\Phi}(M\times [0,T])
	\end{align*}
are bounded.
\end{thm}
\begin{proof}
We will present only the argument for the first mapping property as the second follows along the same lines.

The same argument as in the previous result leads to an equivalent formulation of the statement.
That is, one has to prove that the operator 
$$M(t^{-1/2}x^{-\gamma})\circ \mathbf{H}\circ M(x^{\gamma}):C^\alpha_\Phi(M\times[0,T])\rightarrow C^{1,\alpha}_\Phi(M\times[0,T])$$
is bounded. 
As in the previous theorem, one deduces that the above is equivalent to prove that the operator $\mathbf{G}_t$, defined by $\mathbf{G}_{t}u=V(t^{-1/2}x^{-\gamma}\mathbf{H}x^\gamma)u$ with $V \in \{id,\mathcal{V}_{\phi}\}$, mapping
$$\mathbf{G}_t:C^\alpha_\Phi(M\times[0,T])\rightarrow C^\alpha_\Phi(M\times[0,T])$$
is bounded. 
One has
\[(\mathbf{G}_{t} u)(p,t) = \int_{0}^{t}\int_{M}V((t-\widetilde{t})^{-1/2}H_{\gamma}(t-\widetilde{t},p,\widetilde{p}))u(\widetilde{p},\widetilde{t})\dvol_{\Phi}(\widetilde{p})\di \widetilde{t};\]
where $H_\gamma$ is defined in \eqref{Hgamma}.
The estimates in sections \ref{HDSSection}, \ref{HDTSection} and \ref{SNSection} will already cover the case $V\in\{\Id,\mathcal{V}_{\Phi}\}$. Moreover, it is crucial to note that we are no longer considering elements in $\mathcal{V}_{\Phi}^{2}$, which will lead to an extra $\tau$ term. On the other hand, the term $(t-\widetilde{t})^{-1/2}$ inside the integrand lifts to an extra $\tau^{-1}$ in every region of $M^{2}_{h}$.  
This means that the presence of the term $(t-\widetilde{t})^{-1/2}$ is proportionally compensated by the absence of second order $\Phi$-differential operators (i.e. elements in $\mathcal{V}^2_{\Phi}$).
Thus, in attempting to get these estimates following the same computations as in the upcoming sections, the integrands obtained will have the exact same asymptotics.
\end{proof}

\section{Estimates of H\"{o}lder differences in space} \label{HDSSection}

The aim of this section is to prove the inequality in \eqref{HDS}. 
Consider $p,p'$ to be some fixed points in $M$ and set 
$$M^+=\left\{\widetilde{p}\in M\,\big|\,d(p,\widetilde{p})\le 3d(p,p')\right\},\;\;M^-=\left\{\widetilde{p}\in M\,\big|\,d(p,\widetilde{p})\ge 3d(p,p')\right\}.$$
Let us denote by $V$ any element which is either the identity, a $\Phi$-derivative or a second order $\Phi$-differential operator.
For any $u$ function in $ C^{\alpha}_{\Phi}(M\times [0,T])$, one has 
$${\bf G} u(s,p)-{\bf G} u(s,p')=I_1+I_2+I_{3},$$
where, for $G$ the kernel of $\textbf{G}$, we have
\begin{align*}
&I_1=\int_0^s\int_{M^+}\left[G(s-\widetilde{s},p,\widetilde{p})-G(s-\widetilde{s},p',\widetilde{p})\right]\left[u(\widetilde{s},\widetilde{p})-u(\widetilde{s},p)\right]\dvol_\Phi(\widetilde{p})\di\widetilde{s},\\
&I_2=\int_0^s\int_{M^-}\left[G(s-\widetilde{s},p,\widetilde{p})-G(s-\widetilde{s},p',\widetilde{p})\right]\left[u(\widetilde{s},\widetilde{p})-u(\widetilde{s},p)\right]\dvol_\Phi(\widetilde{p})\di\widetilde{s},\\
&I_3=\int_0^s\int_{M}\left[G(s-\widetilde{s},p,\widetilde{p})-G(s-\widetilde{s},p',\widetilde{p})\right]u(\widetilde{s},p)\dvol_\Phi(\widetilde{p})\di\widetilde{s}.
\end{align*}
Hence, it is clear that \eqref{HDS} is satisfied if $|I_j|\le c\|u\|_\alpha\di(p,p')^\alpha$ for $j=1,2,3$.
\medskip

Since the heat kernel $H$, hence also $G$, is smooth in the interior of $M^{2}_{h}$, the claimed estimates need only to be provided near the boundary hypersurfaces of $M^{2}_{h}$. 
Moreover, as stated in Theorem \ref{MohammadThm}, $\beta^*H$, hence the lift of $G$ as well, vanishes to infinite order away from $\fd\cup\td$ (i.e. in regimes 1 and 2), resulting in trivial estimates.
We will therefore focus solely on the estimates of $I_j$ near $\fd\cup\td$ for every $j=1,2,3$.
In estimating the integrals we will assume, without loss of generality, $G$ to be compactly supported in the regime of interest.
In conclusion, in order to simplify the notation, we will identify the integration regions $M$, $M^+$ and $M^-$ with their lifts. 
Also, we will simply denote by $y$ and $z$ either one or all of the coordinates $y_{i}$ and $z_{j}$ respectively.  The same is true for $y'$ and $z'$.

\subsection{Estimates for $I_2$}

First of all notice that, in this setting, $\widetilde{p}$ is ranging in $M^-$.
We begin by proving the following useful fact.
\begin{lem} \label{ineq}
Let $p''$ be a point in $M$ such that $d(p',p'')\le d(p,p')$. For every point $\widetilde{p}$ in $M^-$, one has
$$\frac{1}{3} d(p,\widetilde{p})\le d(p'',\widetilde{p}).$$
\end{lem}
\begin{proof}
Triangle inequality, the assumption on $p''$ and the fact that $\widetilde{p}$ lies in $M^{-}$, imply
\begin{align*}
d(p,\widetilde{p})&\le d(p,p')+d(p',\widetilde{p})\le d(p,p')+d(p',p'')+d(p'',\widetilde{p})\\
&\le d(p,p')+d(p,p')+d(p'',\widetilde{p})=2d(p,p')+d(p'',\widetilde{p})\\
&\le\frac{2}{3}d(p,\widetilde{p})+d(p'',\widetilde{p}).
\end{align*}
The result follows by cancellation.
\end{proof}
\medskip
To estimate $I_2$, employ the Mean Value Theorem to obtain
\begin{equation*}
\begin{split}
I_2=&|x-x'|\int_0^s\int_{M^-}\partial_\xi G\big|_{(s-\widetilde{s},\xi,y,z,\widetilde{x},\widetilde{y},\widetilde{z})}\left[u(\widetilde{s},\widetilde{x},\widetilde{y},\widetilde{z})-u(\widetilde{s},x,y,z)\right]\dvol_\Phi(\widetilde{p})\di\widetilde{s}\\
+&\|y-y'\|\int_0^s\int_{M^-}\partial_\eta G\big|_{(s-\widetilde{s},x',\eta,z,\widetilde{x},\widetilde{y},\widetilde{z})}\left[u(\widetilde{s},\widetilde{x},\widetilde{y},\widetilde{z})-u(\widetilde{s},x,y,z)\right]\dvol_\Phi(\widetilde{p})\di\widetilde{s}\\
+&\|z-z'\|\int_0^s\int_{M^-}\partial_\zeta G\big|_{(s-\widetilde{s},x',y',\zeta,\widetilde{x},\widetilde{y},\widetilde{z})}\left[u(\widetilde{s},\widetilde{x},\widetilde{y},\widetilde{z})-u(\widetilde{s},x,y,z)\right]\dvol_\Phi(\widetilde{p})\di\widetilde{s}.\\
\end{split}
\end{equation*}
Let $p''$ be a point in the set $\{(\xi,y,z),(x',\eta,z),(x',y',\zeta)\}$.
Clearly $p''$ satisfies either $d(p,p'') \le d(p,p')$ or $d(p',p'') \le d(p,p')$.
If the latter holds, Lemma \ref{ineq} gives $d(p,\widetilde{p}) \le 3d(p'',\widetilde{p})$.
Similarly, arguing by means of the triangle inequality, one sees that the same estimate holds if the former case is satisfied.
In conclusion, for every $\widetilde{p} \in M^{-}$, and for $p''$ as above,
$$d(p,\widetilde{p}) \le 3d(p'',\widetilde{p}).$$
Moreover, since $u$ lies in $C^{\alpha}_{\Phi}(M\times[0,T])$, it follows that
\begin{align*}
I_2\le & c\|u\|_{\alpha}|x-x'|\int_0^s\int_{M^-}\partial_\xi G(s-\widetilde{s},p'',\widetilde{p})d\left(p'',\widetilde{p}\right)^\alpha\dvol_\Phi(\widetilde{p})\di\widetilde{s}\\
&+c\|u\|_\alpha\|y-y'\|\int_0^s\int_{M^-}\partial_\eta G(s-\widetilde{s},p'',\widetilde{p})d\left(p'',\widetilde{p}\right)^\alpha\dvol_\Phi(\widetilde{p})\di\widetilde{s}\\
&+c\|u\|_\alpha\|z-z'\|\int_0^s\int_{M^-}\partial_\zeta G(s-\widetilde{s},p'',\widetilde{p})d\left(p'',\widetilde{p}\right)^\alpha\dvol_\Phi(\widetilde{p})\di\widetilde{s}.
\end{align*}
In the above, with abuse of notation, we denoted by $p''$ any of the occurrences of the point arising from the Mean value theorem.
More precisely, in the first integral $p''$ has coordinates $(\xi,y,z)$, in the second $(x',\eta,z)$ and in the third $(x',y',\zeta)$.
For readability reasons, we will denote the summands in the estimate above respectively by $I_{2,1}$, $I_{2,2}$ and $I_{2,3}$.
Estimates for $I_{2,1}$, $I_{2,2}$ and $I_{2,3}$ can be obtained similarly.
Therefore we present explicit computation only for $I_{2,1}$.
\medskip

The formulae in \eqref{ABHK3} give us the asymptotic behavior of $\partial_{\xi}G$ in this regime.
Using projective coordinates 
$(\tau,x,y,z,\mathcal{S}',\mathcal{U}',\mathcal{Z}')$ given by
\[\mathcal{S}' = \dfrac{\widetilde{x}-\xi}{\xi^{2}}, \hspace{2mm} \mathcal{U}' = \dfrac{\widetilde{y}-y}{\xi}, \hspace{2mm} \mathcal{Z}' = \widetilde{z}-z \hspace{2mm} \mbox{and} \hspace{2mm} \tau = \sqrt{t - \widetilde{s}},\]
one has  
\begin{equation}\label{InequalityForI'_2withp''}
|I_{2,1}|\le c\|u\|_\alpha|x-x'|\int_{0}^{\sqrt{s}}\int_{M^-}\sigma^{-m-2}\xi^{-2}G_{0}\beta^*(d(p'',\widetilde{p})^\alpha)
\di\mathcal{S}'\di\mathcal{U}'\di\mathcal{Z}'\di \tau
\end{equation}
with $G_{0}$ being bounded. 

\medskip
Let us now analyse the distance $\beta^*(d(p'',\widetilde{p})^\alpha)$.
Recall that, near $\fd\cup \td$, $\xi \sim \widetilde{x}$.  
This implies, in particular, $\widetilde{\mathcal{S}''} = \widetilde{x} / \xi \sim 1$ 
and thus giving
\begin{align*}
    d((\xi,y,z),(\widetilde{x},\widetilde{y},\widetilde{z}))&= 
    \sqrt{|\xi-\tilde{x}|^{2}+(\xi+\tilde{x})^{2}\|y-\tilde{y}\|^{2}+(\xi+\tilde{x})^{4}\|z-\tilde{z}\|^{2}} \\
    & \sim \xi^2 \sqrt{ |\mathcal{S}'|^2 + |\mathcal{U}'|^2 + |\mathcal{Z}'|^2 } \\
    &=:\xi^2 r(\mathcal{S}',\mathcal{U}',\mathcal{Z}').
\end{align*}
Note that the function $r$ is nothing but the radial distance in polar coordinates from the origin.  
From the above we conclude the existence of some constant $c$ such that
\begin{equation}
\beta^*\left(d\left((\xi,y,z),(\widetilde{x},\widetilde{y},\widetilde{z})\right)^\alpha\right)\le c(\xi^2 r)^{\alpha}.
\end{equation}
By using $r$ as the radial coordinate in $\mathbb{R}^{m}$ we can perform a change of coordinates in \eqref{InequalityForI'_2withp''}, leading to
$$|I_{2,1}|\le c\|u\|_\alpha|x-x'|\int_{0}^{\sqrt{s}}\int_{M^{-}} \sigma^{-m-2}\xi^{-2+2\alpha}r^{m-1+\alpha}G_{0}\di r\di \mbox{(angle)}\di\tau. $$
Finally, setting $\sigma=r^{-1}\tau$, it follows that the asymptotic behavior of $\varsigma^{-1}$ is (cf. \eqref{coordiantesneartd})
$$\varsigma^{-1}\sim \sqrt{|\mathcal{S}|^2+\|\mathcal{U}\|^2+\|\mathcal{Z}\|^2}.$$
This implies that integrating $G_{0}$ against any negative power of $\varsigma$ leads to a bounded term. 
Moreover, for $r$ defined as above,  $M^{-} \subset \{\xi^{-2}d(p,p') \le cr\}$ for some constant $c>0$.  
Thus, once the angular variables are being integrated out, it follows, by performing yet another change of coordinates given by $\tau\mapsto\sigma$,
\begin{align*}
|I'_2|&\le c\|u\|_\alpha|x-x'|\int_{\xi^{-2} d(p,p')}^\infty r^{-2+\alpha}\xi^{-2+2\alpha}\di r\\
&=c\|u\|_{\alpha} |x-x'|\xi^{-2+2\alpha}(\xi^{-2}d(p,p'))^{-1+\alpha}\\
&\le c\|u\|_\alpha d(p,p')^{\alpha},
\end{align*}
as claimed.

\subsection{Estimates of $I_{1}$}

As for the estimates of $I_2$, we see that
\begin{align*}
I_{1} = &\int_{0}^{s}\int_{M^{+}} G(s-\widetilde{s},p,\widetilde{p})[u(\widetilde{p},\widetilde{s})-u(p,\widetilde{s})]\dvol_{\Phi}(\widetilde{p})\di \widetilde{s} \\ 
&- \int_{0}^{s}\int_{M^{+}} G(s-\widetilde{s},p',\widetilde{p})[u(\widetilde{p},\widetilde{s})-u(p',\widetilde{s})]\dvol_{\Phi}(\widetilde{p})\di \widetilde{s} \\
&+ \int_{0}^{s}\int_{M^{+}} G(s-\widetilde{s},p',\widetilde{p})[u(p,\widetilde{s})-u(p',\widetilde{s})]\dvol_{\Phi}(\widetilde{p})\di \widetilde{s} \\
=: &I_{1,1} - I_{1,2} + I_{1,3}.
\end{align*}
Clearly the estimates for $I_{1,1}$ and $I_{1,2}$ will be similar, thus we present the full computations only for $I_{1,1}$.
Moreover, as for $I_2$, estimates away from $\fd\cup\td$ are trivial thus we will focus on the estimates near $\fd \cup \td$.
%

\subsubsection*{\textbf{Estimate of $I_{1,1}$}}

Near $\fd \cup \td$, the asymptotics of $G$ are given by the expression in \eqref{ABHK3}.  In particular  $G\sim\sigma^{-m-2}G_{0}$, with $G_{0}$ vanishing to infinite order when near $\|(\mathcal{S},\mathcal{U},\mathcal{Z})\|\rightarrow \infty$.
Let us choose projective coordinates given by $(\tau,x,y,z,\mathcal{S}',\mathcal{U}',\mathcal{Z}')$, with 
\[\mathcal{S}' = \dfrac{\widetilde{x}-x}{x^{2}}, \hspace{2mm} \mathcal{U}' = \dfrac{\widetilde{y}-y}{x}, \hspace{2mm} \mathcal{Z}' = \widetilde{z}-z \hspace{2mm} \mbox{and} \hspace{2mm} \tau = \sqrt{s - \widetilde{s}}.\]
Recall that, in these projective coordinates, the lift of the volume form is expressed as in \eqref{SecondBlowUpProjCoord1}, resulting in
\begin{align*}
|I_{1,1}| &\le \|u\|_{\alpha}\int_{0}^{\sqrt{s}}\int_{M^{+}} \sigma^{-m-1}G_{0}\beta^{*}d((x,y,z),(\widetilde{x},\widetilde{y},\widetilde{z}))^{\alpha} \di \mathcal{S}' \di \mathcal{U}' \di \mathcal{Z}' \di \sigma. \\
\end{align*}
Furthermore, near $\fd\cup\td$, $x\sim \widetilde{x}$.  
Thus, as already done for the estimates for $I_2$, set $r(\mathcal{S}',\mathcal{U}',\mathcal{Z}') := \sqrt{|\mathcal{S}'|^{2}+\|\mathcal{U}'\|^{2}+\|\mathcal{Z}'\|^{2}}$. 
This implies 
\begin{align*}
    \beta^{*}d((x,y,z),(\widetilde{x},\widetilde{y},\widetilde{z})) &= \sqrt{|x-\tilde{x}|^{2}+(x+\tilde{x})^{2}\|y-\tilde{y}\|^{2}+(x+\tilde{x})^{4}\|z-\tilde{z}\|^{2}}\\
    &\sim x^{2}\sqrt{|\mathcal{S}'|^{2}+\|\mathcal{U}'\|^{2}+\|\mathcal{Z}'\|^{2}}\\
    &= c x^{2} r(\mathcal{S}',\mathcal{U}',\mathcal{Z}').
\end{align*}
It follows, $M^{+} = \{r \le cx^{-2}d(p,p')\}$ for some constant $c>0$.

\medskip
Let us now denote $\varsigma = \tau/r$. 
Arguing as in the estimates for $I_2$, $r$ can be thought as the radial distance in $\mathbb{R}^m$ with coordinates given by $
(\mathcal{S}',\mathcal{U}',\mathcal{Z}')$.
We can therefore consider polar coordinates and perform a change of coordinates in the integral above.
Integrating out once again the angular coordinates, we find
\[|I_{1,1}| \le c\|u\|_{\alpha}x^{2\alpha}\int_{I(\varsigma)}\int_{0}^{x^{-2}d(p,p')} \varsigma^{-m-1}r^{-1+\alpha}G_{0} \di r \di \sigma.\]
The estimate now follows by noticing $\varsigma^{-m-1}G_{0}$ to be bounded (due to the decay properties of $G_{0}$).

\subsubsection*{\textbf{Estimate of $I_{1,3}$}}

As mentioned earlier, the estimates for $I_{1,3}$ follow along slightly different lines from the one for $I_{1,1}$, since it follows from employing integration by parts.
First of all we consider projective coordinates $(\tau,x,y,z,\mathcal{S},\mathcal{U},\mathcal{Z})$ with
$$\mathcal{S}=\frac{\widetilde{x}-x}{\tau x^2},\;\;\mathcal{U}=\frac{y-\widetilde{y}}{\tau x},\;\;\mathcal{Z}=\frac{z-\widetilde{z}}{\tau}.$$
Note that the ``worst case scenario'' for $I_{1,3}$ is given by $G = \sigma^{-m-2}(V_{1}V_{2}G_{0})$ with both $V_{1},V_{2} \in \{\partial_{\mathcal{S}},\partial_{\mathcal{U}},\partial_{\mathcal{Z}}\}$.
For the sake of simplicity, since the general case is similar, we assume $V_{1} = \partial_{\mathcal{S}}$.  
On the other hand, for fixed $(\tau,\mathcal{U},\mathcal{Z})$ one has $M^{+} = \{|\mathcal{S}| \le r(\tau,\mathcal{U},\mathcal{Z})\}$, where this expression for $M^{+}$ comes from the fact that $r$ is taken originally in terms of $(\mathcal{S}',\mathcal{U}',\mathcal{Z}') = (\tau\mathcal{S},\tau\mathcal{U},\tau\mathcal{Z})$. 
Hence
\[\beta^{*}(\dvol_{\Phi}(\widetilde{p})\di \widetilde{s}) = h(x+x^{2}\tau\mathcal{S},y+x\tau\mathcal{U},z+\tau\mathcal{Z}) \sigma^{m+1}\di \mathcal{S} \di \mathcal{U} \di \mathcal{Z} \di \tau,\]
for $h$ being smooth function.
Moreover, let us denote $\delta u:=[u(p,\widetilde{s})- u(p',\widetilde{s})]$.
Clearly $\delta u$ does not depend on $\widetilde{p}$, meaning that $I_{1,3}$ can be written as
\begin{align*}
I_{1,3} &= \int_{0}^{\sqrt{s}}\delta u\int_{M^{+}} \sigma^{-1}(\partial_{\mathcal{S}}V_{2}G_{0})h \di \mathcal{S} \di \mathcal{U}\di \mathcal{Z} \di \tau\\
&= \int_{0}^{\sqrt{s}}\delta u\int_{\partial M^{+}} \sigma^{-1}(V_{2}G_{0})\big|_{|\mathcal{S}| = r}h \di \mathcal{S} \di \mathcal{U}\di \mathcal{Z} \di \tau\\
&\hspace{4mm} - \int_{0}^{\sqrt{s}}\delta u\int_{ M^{+}} \sigma^{-1}(V_{2}G_{0})\partial_{\mathcal{S}}h \di \mathcal{S} \di \mathcal{U}\di \mathcal{Z} \di \tau\\
&=: I^{1}_{1,3} - I^{2}_{1,3}.
\end{align*}
We begin by noticing that the derivative of the smooth function $h$ with respect to $\mathcal{S}$ can be written as $\partial_{\mathcal{S}}h = x^2\sigma h'$.
The $\sigma$ appearing in this derivative cancels with the $\sigma^{-1}$.
Thus, since $x(V_{2}G_0)$ is bounded, the whole integral over $M^+$ in $I^2_{1,3}$ is bounded.  
The estimate follows by estimating $\delta u$ against the $\alpha$-norm of $u$ multiplied by the distance $d(p,p')^\alpha$.
\medskip

Let us now focus on the integral $I^{1}_{1,3}$.
For simplicity we will denote $V_{2}G_{0}$ just by $G'_{0}$. 
We begin by performing a change of coordinates in $I^1_{1,3}$.
In particular, we choose the projective coordinates $(\tau,x,y,z,\mathcal{S}',\mathcal{U}',\mathcal{Z}')$ with
$$\mathcal{S}'=\frac{\widetilde{x}-x}{x^2},\;\;\mathcal{U}'=\frac{\widetilde{y}-y}{x},\;\;\mathcal{Z}'=\widetilde{z}-z.$$
This accounts into a change of coordinates of the form $\mathcal{S}'=\tau\mathcal{S}$, \;$\mathcal{U}'=\tau\mathcal{U}$ and $\mathcal{Z}'=\tau\mathcal{Z}$, leading to 
\begin{align*}
|I^{1}_{1,3}| &\le \|u\|_{\alpha}\int_{0}^{\sqrt{s}}\int_{\partial M^{+}} \sigma^{-m}(G'_{0})\big|_{|\mathcal{S}| = r}\beta^{*}d((x,y,z),(\widetilde{x},\widetilde{y},\widetilde{z}))^{\alpha} \di U \di Z \di \tau.
\end{align*}
We proceed exactly as for the estimates of $I_{1,1}$.
This means we consider polar coordinates in $\mathbb{R}^{m-1}$ (with coordinates $(\mathcal{U}',\mathcal{Z}')$) and denote by $R$ the radial component.
Next, we set $\varsigma = \sigma/R$.
Furthermore, from the expression of $M^+$ above, one has $\{|\mathcal{S}|=r\}=\partial M^+=\{\widetilde{p}\,|\, d(p,\widetilde{p}) = 3d(p,p')\}$. 
In particular, taking $\widetilde{p}$ a point lying in the boundary of $\partial M^*$ of $M^+$,
\[2d(p,p') \le d(p,\widetilde{p}) \le 4d(p,p').\]
Integrating out the angular component results in 
\begin{align*}
|I^{1}_{1,3}| &\le \|u\|_{\alpha}\int_{0}^{\infty}\int_{0}^{4d(p,p')} R^{-1+\alpha}\varsigma^{-m}\left(\sqrt{\dfrac{|\mathcal{S}'|^{2}+\|\mathcal{U}'\|^{2}+\|\mathcal{Z}'\|^{2}}{\|\mathcal{U}'\|^{2}+\|\mathcal{Z}'\|^{2}}}\right)^{\alpha}(G'_{0})\big|_{|\mathcal{S}| = r} \di R \di \sigma \\
& \le c\|u\|_{\alpha}d(p,p')^{\alpha}.
\end{align*}
\begin{rmk}
Notice that, in the case $V = \Id$, the asymptotic behavior of the integrand will have an improvement by $\sigma^{2}$.  
This means, in particular, that integration by part will no longer be necessary. 
\end{rmk}

\subsection{Estimates of $I_{3}$}

As usual, let us assume $p=(x,y,z)$ and $p'=(x',y',z')$.
By adding and subtracting the same quantity, we can express $I_3$ as the sum of three integrals $I_{3,1}$, $I_{3,2}$ and $I_{3,3}$ as displayed below:
\begin{align*}
    I_{3} = &\int_{0}^{s}\int_{M}[G(s-\widetilde{s},p,\widetilde{p})-G(s-\widetilde{s},(x',y,z),\widetilde{p})]u(\widetilde{s},p)\dvol_{\Phi}(\widetilde{p})\di \widetilde{s} \\
    &+\int_{0}^{s}\int_{M}[G(s-\widetilde{s},(x',y,z),\widetilde{p})-G(s-\widetilde{s},(x',y',z),\widetilde{p})]u(\widetilde{s},p)\dvol_{\Phi}(\widetilde{p})\di \widetilde{s} \\
    &+\int_{0}^{s}\int_{M}[G(s-\widetilde{s},(x',y',z),\widetilde{p})-G(s-\widetilde{s},p',\widetilde{p})]u(\widetilde{s},p)\dvol_{\Phi}(\widetilde{p})\di \widetilde{s} \\
    =: &I_{3,1} + I_{3,2} + I_{3,3}.
\end{align*}
Our first aim is to show that the expression above can actually be reduced to $I_3=I_{3,1}$.
To see this we show that both $I_{3,2}$ and $I_{3,3}$ are vanishing.
Due to similarity we will prove only $I_{3,2}=0$.
\medskip

Recall that, as showed in \S \ref{StochPhiMflds}, $\Phi$-manifolds are stochastically complete.
Thus, by recalling the expression for the integral kernel $G$, it follows that $I_{3,2}$ can be written as
\begin{equation*}
\begin{split}
    I_{3,2} =& \int_{0}^{s}x^{\gamma}u(p,\widetilde{s})\left(\int_{M}[V(x^{-\gamma}H)(s-\widetilde{s},(x',y,z),\widetilde{p})\right. \\
    &-V(x^{-\gamma}H)(s-\widetilde{s},(x',y',z),\widetilde{p})]\dvol_{\Phi}(\widetilde{p})\bigg)\di\widetilde{s}\\
    =& \int_{0}^{s}((x')^{-\gamma}-(x')^{-\gamma})x^{\gamma}u(p,\widetilde{s})\di \widetilde{s} = 0.
\end{split}
\end{equation*}
We can now estimate $I_3$.
As for the previous cases, estimates away from $\fd\cup\td$ are trivial thus will not be presented here.
\medskip

The Mean Value Theorem allows to rewrite $I_{3}$ as follows:
\begin{align}
I_{3} &= |x-x'|\int_{0}^{s}\int_{M} \partial_{\xi}G\big|_{(s-\widetilde{s},\xi,y,z,\widetilde{x},\widetilde{y},\widetilde{z})}u(x,y,z,\widetilde{s})\dvol_{\Phi}(\widetilde{p})\di \widetilde{s}
\end{align}
In projective coordinate $(\tau,\xi,y,z,\mathcal{S},\mathcal{U},\mathcal{Z})$, with
$$\mathcal{S}=\frac{\widetilde{x}-\xi}{\tau \xi^2},\;\;\mathcal{U}=\frac{y-\widetilde{y}}{\tau \xi},\;\;\mathcal{Z}=\frac{z-\widetilde{z}}{\tau},$$
one finds that the lifted vector field $\beta^*(\partial_\xi)$ admits an expression of the form
\[\beta^{*}(\partial_{\xi}) = \partial_{\xi} - [2\xi^{-1}\mathcal{S} + \xi^{-2}\sigma^{-1}]\partial_{\mathcal{S}} - \xi^{-1}\mathcal{U}\partial_{\mathcal{U}}.\]
In particular, we conclude $\beta^{*}(\partial_{\xi}G) \sim \xi^{-2}\sigma^{-1}\partial_{\mathcal{S}}G'_{0} $.  
Note that the asymptotic behavior of $G'_{0}$ are similar to those of $G_0$ in $I_{1,3}$ near ${\td}$.
Furthermore, since the function $u(p,\widetilde{s})$ is constant with respect to the spatial integration (i.e. with respect to $\widetilde{p}$), integration by parts gives  
\begin{align*}
\int_{0}^{s}\int_{M}& \xi^{-2}\sigma^{-1}\partial_{\mathcal{S}}G'_{0} u(p,s-\sigma^{2}) h \di \mathcal{S} \di \mathcal{U}\di \mathcal{Z} \di \sigma\\
&=\int_{0}^{s}\int_{\partial_M} \xi^{-2}\sigma^{-1}G'_{0}\big|_{|\mathcal{S}|=\infty} u(p,s-\sigma^{2}) h \di \mathcal{S} \di \mathcal{U}\di \mathcal{Z} \di \sigma\\
&\hspace{4mm}-\int_{0}^{t}\int_{M} \xi^{-2}\sigma^{-1}G'_{0} u(p,s-\sigma^{2}) \partial_{\mathcal{S}}h \di \mathcal{S} \di \mathcal{U}\di \mathcal{Z} \di \sigma.
\end{align*}
Due to the decay properties of $H_{\gamma}$ near $\partial \overline{M}$ (cf. \eqref{ABHK3}), the integral along the boundary vanishes.
\medskip
  
Similarly to the estimates of $I^2_{1,3}$, we find $\partial_{\mathcal{S}} h = \xi^{2}\sigma h'$.
This implies $\tau\xi^2$ and $\tau^{-1}\xi^{-2}$ cancel out leading to 
\begin{align*}
\int_{0}^{s}\int_{M}& \xi^{-2}\tau^{-1}\partial_{\mathcal{S}}G'_{0} u(p,s-\tau^{2}) h \di \mathcal{S} \di \mathcal{U}\di \mathcal{Z} \di \tau\\
&=-\int_{0}^{s}\int_{M} G'_{0} u(p,s-\tau^{2}) h' \di \mathcal{S} \di \mathcal{U}\di \mathcal{Z} \di \tau.
\end{align*}
The estimate now follows by proceeding as for the estimates of $I^2_{1,3}$.
It is important to point out that, contrarily  to $I^2_{1,3}$, the boundary term is vanishing.
\medskip

With this last inequality, we conclude the proof of H\"{o}lder differences in space.

\section{Estimates for H\"{o}lder differences in time}\label{HDTSection}

In this section, we will prove the estimates stated in \eqref{HDT}.
Without loss of generality, we can consider $s< s'$. 
Indeed, in order to gain the estimates for $s'<s$, it will be enough to repeat all the upcoming estimates interchanging the role of $s$ and $s'$.
Moreover, we begin by proving the estimates under the initial further assumption $2s' - s \geqslant 0$ (i.e., $s' < s \le 2s'$).
At the end of the section, we will explain how to proceed for the other case, i.e. $2s'-s<0$.
\medskip

Let $T_-$, $T_+$ and $T'_+$ denote the intervals
\[T_{-} := [0,2s'-s], \hspace{2mm} T_{+} := [2s'-s,t] \hspace{2mm} \mbox{and} \hspace{2mm} T'_{+} := [2s'-s,s'].\]
As it has already been done for the estimates of H\"{o}lder differences in space (cf. \S \ref{HDSSection}), we denote by $\mathbf{G}$ the operator $\mathbf{G}=V\mathbf{H}_{\gamma}$ for $V \in \{\mbox{id}\}\cup \mathcal{V}_{\Phi}\cup \mathcal{V}^{2}_{\Phi}$.
Using the same argument Bahuaud and Vertman used in \cite[\S 3.2]{bahuaud2014yamabe}, we deduce
\begin{align*}
\mathbf{G} u(p,s) - \mathbf{G} u(p,s')
=  &|s-s'|\int_{T_{-}}\int_{M} \partial_{\theta}G\big|_{(\theta-\widetilde{s},p,\widetilde{p})}[u(\widetilde{p},\widetilde{s})-u(p,\widetilde{s})] \dvol_{\Phi}(\widetilde{p})\di \widetilde{s} \vspace{1mm}\\
&+ \int_{T_{+}}\int_{M} G(s-\widetilde{s},p,\widetilde{p})[u(\widetilde{p},\widetilde{s})-u(p,\widetilde{s})]\dvol_{\Phi}(\widetilde{p})\di \widetilde{s} \\
&- \int_{T'_{+}}\int_{M} G(s'-\widetilde{s},p,\widetilde{p})[u(\widetilde{p},\widetilde{s})-u(p,\widetilde{s})]\dvol_{\Phi}(\widetilde{p})\di \widetilde{s}\\
&+ \int_{0}^{s}\int_{M}G(s-\widetilde{s},p,\widetilde{p})u(p,\widetilde{s})\dvol_{\Phi}(\widetilde{p})\di \widetilde{s}\\
&- \int_{0}^{s'}\int_{M}G(s'-\widetilde{s},p,\widetilde{p})u(p,\widetilde{s})\dvol_{\Phi}(\widetilde{p})\di \widetilde{s}\\
=: &L_{1} + L_{2} - L_{3} + L_{4} - L_{5}
\end{align*}

First of all, notice that space-variable $p$ is constant in the integration.
Second, as pointed out in \S \ref{StochPhiMflds}, $(M,g_{\Phi})$ is stochastically complete, i.e. 
$$\int_M H(s,p,\widetilde{p})\dvol_{\Phi}(\widetilde{p})=1.$$
These two key observations put together gives us the following estimate:
\begin{align*}
    L_{4}-L_{5} = \int_{0}^{s}u(p,\widetilde{s})\di \widetilde{s} - \int_{0}^{s'}u(p,\widetilde{s})\di \widetilde{s}
    \le C\|u\|_{\infty}|s-s'|^{\alpha/2}.
\end{align*}
Therefore, it is clear that in order to prove the inequality claimed in \eqref{HDT}, it is only necessary to estimate $L_{1}, L_{2}$ and $L_{3}$.
However, due to similarities between the terms $L_{2}$ and $L_{3}$,
we will only present one of them.
In conclusion, in what follows we will present the estimates, at each regime, of the terms $L_1$ and $L_2$.
Moreover, as for the H\"{o}lder differences in space, estimates away from $\fd\cup \td$ are trivial and will, therefore, be omitted.

\subsection{Estimates for $L_{1}$}

In projective coordinates $(\tau,x,y,z,\mathcal{S},\mathcal{U},\mathcal{Z})$ with
$$\mathcal{S}=\frac{\widetilde{x}-x}{\tau x^2},\;\;\mathcal{U}=\frac{y-\widetilde{y}}{\tau x},\;\;\mathcal{Z}=\frac{z-\widetilde{z}}{\tau},$$  
the asymptotics of $H_{\gamma}$ near $\fd\cup\td$ are given by \eqref{ABHK3}.
In particular, it follows $\beta^{*}\partial_{\theta}G \sim \sigma^{-m-4}G'_{0}$, with $G'_{0}$ being polyhomogeneous and vanishing to infinite order when $\|(\mathcal{S},\mathcal{U},\mathcal{Z})\|\rightarrow \infty$.
Also, the volume form has asymptotics of the form 
\begin{equation} \label{volumetd}
\beta^{*}(\dvol_{\Phi}(\widetilde{p})\di \widetilde{s}) \sim \tau^{m+1}h \di\mathcal{S}\di\mathcal{U}\di\mathcal{Z}\di\tau,
\end{equation}
where $h$ is a smooth function of $\widetilde{p} = (x+x^{2}\mathcal{S}\sigma,y+x\mathcal{U}\sigma,z+\sigma\mathcal{Z})$ (cf. \eqref{ABHK3}).
Further, recall that, near $\fd\cup\td$, $x\sim \widetilde{x}$.
In particular this implies
\begin{equation} \label{rmiddle}
d(p,\widetilde{p}) \le c\tau\rho_{\fd}\sqrt{|\mathcal{S}|^{2} + \|\mathcal{U}\|^{2} + \|\mathcal{Z}\|^{2}} =: c\tau x r(\mathcal{S},\mathcal{U},\mathcal{Z}).
\end{equation}
Note that, due to its expression, $r$ is bounded as long as its entries are bounded.  Therefore $G'_{0}r^{\alpha}$ is bounded everywhere. 
Moreover, for every $\widetilde{s}\in T_{-}$, $|\theta-\widetilde{s}| \ge |s-s'|$.  In conclusion the integral $L_1$, can be estimated as follows
\begin{align*}
|L_{1}| & \le  \|u\|_{\alpha}|s-s'|\int_{T_{-}} \int_{M} |\tau^{-3}G'_{0}d(p,\widetilde{p})^{\alpha}|\di\mathcal{S} \di\mathcal{U} \di\mathcal{Z} \di\tau  \\
& \le  C\|u\|_{\alpha}|s-s'|\int_{\sqrt{s-s'}}^{\infty} \int_{M} |\tau^{-3+\alpha}x^{\alpha}G'_{0}r^{\alpha}|\di\mathcal{S} \di\mathcal{U} \di\mathcal{Z} \di\tau  \\
& \le  C\|u\|_{\alpha}|s-s'|^{\alpha/2},
\end{align*}
thus completing the estimates for $L_1$.

\subsection{Estimates for $L_{2}$}

Proceeding exactly as for $L_1$, in the same projective coordinates $(\tau,x,y,z,\mathcal{S},\mathcal{U},\mathcal{Z})$, we recall, from \eqref{ABHK3}, that the asymptotics of $\beta^*G$ are given by $\beta^{*}(G) \sim  \tau^{-m-2}G'_{0}$, with $G'_{0}$ being polyhomogeneous and vanishing to infinite order whenever $\|(\mathcal{S}, \mathcal{U},\mathcal{Z})\|\rightarrow \infty$.
Furthermore, in these coordinates, the lift of the volume form has the expression as in \eqref{volumetd}.
Therefore, by expressing $L_2$ in projective coordinates one finds
\begin{align*}
|L_{2}| & \le  \|u\|_{\alpha}\int_{T_{+}}\int_{M} |\tau^{-1}G'_{0}d(p,\widetilde{p})^{\alpha}|\di\varsigma \di\eta \di\zeta \di\tau \\
& \le  c \|u\|_{\alpha}\int_{T_{+}}\int_{M} |\tau^{-1+\alpha}G'_{0}r^{\alpha}|\di\mathcal{S} \di\mathcal{U} \di\mathcal{Z} \di\tau \\
& \le  c \|u\|_{\alpha}|s-s'|^{\alpha/2},
\end{align*}
concluding the estimates for the $L_{2}$-term.
This completes the estimates for time difference with derivatives under the assumption that $2s'-s\geqslant 0$. 
\medskip

Let us now remove the initial further assumption, meaning that we assume $2s' - s< 0$.  
From the assumption $s > s'\ge 0$ it follows that $2s'<3s$, allowing us to conclude $s'<s <2|s-s'|$.
Therefore, the H\"{o}lder differences in time, under the assumption $2s'-s<0$, can be expressed as
\begin{align*}
\mathbf{G}u(p,s) - \mathbf{G}u(p,s') 
= &\int_{0}^{s}\int_{M} G(s-\widetilde{s},p,\widetilde{p})[u(\widetilde{p},\widetilde{s}) - u(p,\widetilde{s})]\dvol_{\Phi}(\widetilde{p})\di \widetilde{s} \\
&- \int_{0}^{s'}\int_{M} G(s'-\widetilde{s},p,\widetilde{p})[u(\widetilde{p},\widetilde{s}) - u(p,\widetilde{s})]\dvol_{\Phi}(\widetilde{p})\di \widetilde{s}.
\end{align*}
The integrals above can be treated similarly to $L_{2}$.

\begin{rmk}
	This completes the estimates for time differences with derivatives under the assumption $t' < t$.  However, to obtain the estimates for the case $t < t'$, one just need to interchange the roles of $t$ and $t'$.
\end{rmk}
The proof of the estimate claimed in \eqref{HDT} is therefore complete

\section{Estimates for the supremum norm } \label{SNSection}

We finally reached the final step for the proof of Theorem \ref{MappingPropertiesTHM}; consisting in obtaining the estimates for the supremum of $\mathbf{G}u$ claimed in \eqref{SN}.  
For $(p,s)$ in $M\times[0,T]$, we write $\mathbf{G}u(p,s)$ as
\begin{align*}
\mathbf{G}u(p,s) =&  \int_{0}^{s}\int_{M} G(s-\widetilde{s},p,\widetilde{p})u(\widetilde{p},\widetilde{s})\dvol_{\Phi}(\widetilde{p})\di \widetilde{s} \vspace{1mm} \\
=& \int_{0}^{s}\int_{M} G(s-\widetilde{s},p,\widetilde{p})[u(\widetilde{p},\widetilde{s})-u(p,\widetilde{s})]\dvol_{\Phi}(\widetilde{p})\di \widetilde{s} \vspace{1mm} \\
&+\int_0^s\int_M G(s-\widetilde{s},p,\widetilde{p})u(p,\widetilde{s})\dvol_{\Phi}(\widetilde{p})\di \widetilde{s}\\
& = J_1+J_2.
\end{align*}
As for the previous cases, the estimates away from $\fd\cup\td$ are trivial hence they will be omitted.
\medskip

We begin by estimating $J_1$.
In projective coordinates $(\tau,x,y,z,\mathcal{S},\mathcal{U},\mathcal{Z})$, described by 
$$\mathcal{S}=\frac{\widetilde{x}-x}{\tau x^2},\;\;\mathcal{U}=\frac{y-\widetilde{y}}{\tau x},\;\;\mathcal{Z}=\frac{z-\widetilde{z}}{\tau},$$
and in view of \eqref{ABHK3}, one has that the integrand has asymptotic behavior described by
$$\beta^*(G(s-\widetilde{s},p,\widetilde{p})\dvol_{\Phi}(\widetilde{p})\di \widetilde{s})=\textcolor{red}{\tau}^{-1}G'_0 \di \mathcal{S}\di\mathcal{U}\di\mathcal{Z}\di \textcolor{red}{\tau}.$$
Where $G_0$ vanishes to infinite order as $\|(\mathcal{S},\mathcal{U},\mathcal{Z})\|\rightarrow\infty$.
Recall that, near $\fd\cup \td$, $x\sim \tilde{x}$; thus
\begin{align*}
d(p,\widetilde{p})&\sim \tau x r(\mathcal{S},\mathcal{U},\mathcal{Z}),
\end{align*}
where $r=\sqrt{|\mathcal{S}|^2+\|\mathcal{U}\|^2+\|\mathcal{Z}\|^2}$ is bounded for as long as its entries are bounded.
Therefore $J_1$ can be estimated as 
\begin{align*}
|J_{1}| & \le C  \|u\|_{\alpha} \int_{0}^{\sqrt{s}}\int_{M} \tau^{-1}G'_{0}d(p,\widetilde{p})^{\alpha}\di\mathcal{S} \di\mathcal{U} \di\mathcal{Z} \di\tau \\ 
& =  C\|u\|_{\alpha} \int_{0}^{\sqrt{s}}\int_{M}\left|\tau^{-1+\alpha}G'_{0}r(\mathcal{S},\mathcal{U},\mathcal{Z})^{\alpha}\right|\di\mathcal{S} \di\mathcal{U} \di\mathcal{Z} \di\tau \\ 
& \le  C\|u\|_{\alpha}s^{\alpha/2};
\end{align*}
which proves the claimed inequality.
\medskip

The estimate in \eqref{SN} now follows if $J_2$ satisfies a similar estimate to the one for $|J_1|$ above. 
To see this one may argue exactly as we have done for the estimates of $L_4$ in \S \ref{HDTSection}; that is using the fact that $\Phi$-manifolds are stochastically complete and estimating $|u(p,\widetilde{s})|\le \|u\|_{\infty}\le \|u\|_{\alpha}$.

\section{Short-time existence and regularity of solutions}\label{ShortTimeExistenceSec}

We finish this paper by presenting a direct application of the mapping properties in Theorem \ref{mappropH}.
Namely we prove that the following Cauchy problem for a semi-linear heat equation 
\begin{equation}\label{cauchyshort}
(\partial_{t} + \Delta_{\Phi})u = F(u), \;\; u|_{t=0}=0;
\end{equation}
with the operator $F$ subject to some restrictions, on $\Phi$-manifolds admits a unique solution for short time.
\medskip

First let us lay down the suitable assumptions for the operator $F$:
\begin{enumerate}
\item $F:x^{\gamma}C^{k+2,\alpha}_{\Phi}(M\times [0,T])\rightarrow C^{k,\alpha}_{\Phi}(M\times [0,T])$;
\item $F$ can be written as a sum $F = F_{1} + F_{2}$ with 
\begin{itemize}
\item[i)]$F_{1}:x^{\gamma}C^{k+2,\alpha}_{\Phi}\rightarrow x^{\gamma}C^{k+1,\alpha}_{\Phi}(M\times [0,T]),$
\item[ii)] $F_{2}:x^{\gamma}C^{k+2,\alpha}_{\Phi}\rightarrow x^{\gamma}C^{k,\alpha}_{\Phi}(M\times [0,T]);$
\end{itemize}
\item Let $u,u' \in x^{\gamma}C^{k+2,\alpha}_{\Phi}(M\times [0,T])$ have $(k+2,\alpha,\gamma)$-norm bounded from above by some $\eta>0$; that is $\|u\|_{k+2,\alpha,\gamma},\|u'\|_{k+2,\alpha,\gamma} \le \eta$.
Then there exists some $C_{\eta}>0$ such that
\begin{itemize}
\item[i)] $\|F_{1}(u) - F_{1}(u')\|_{k+1,\alpha,\gamma} \le C_{\eta}\|u-u'\|_{k+2,\alpha,\gamma}$, $\|F_{1}(u)\|_{k+1,\alpha,\gamma} \le C_{\eta}\|u\|_{k+2,\gamma,\alpha},$
\item[ii)] $\|F_{2}(u) - F_{2}(u')\|_{k,\alpha,\gamma} \le C_{\eta}\max\{\|u\|_{k+2,\alpha,\gamma},\|u'\|_{k+2,\alpha,\gamma}\}\|u-u'\|_{k+2,\alpha,\gamma}$, \newline $\|F_{2}(u)\|_{k,\alpha,\gamma} \le C_{\eta}\|u\|^{2}_{k+2,\alpha,\gamma}.$
\end{itemize}
\end{enumerate}
\medskip
Note that, due to Theorem \ref{mappropH}, if there exists some $u^*\in x^\gamma C^{k+2,\alpha}_{\Phi}(M\times[0,T])$ so that 
$$\mathbf{H}\left(F(u^*)\right)=u^*,$$
then $u^*$ is a solution of \eqref{cauchyshort}.
We have successfully transformed the problem of finding a solution to a non-linear Cauchy problem of the form \eqref{cauchyshort} into a fixed-point problem.
Existence of fixed-points of maps on Banach spaces is guaranteed by Banach's fixed-point Theorem.
Thus, the core of the proof of Corollary \ref{theorem4}, which we will recall here for convenience of the reader, will consist on an application of Banach's fixed-point Theorem.
\begin{thm}[Corollary \ref{theorem4}]
Consider the Cauchy problem for a semi-linear heat equation 
\begin{equation}\label{CauchyProbNonLin1}
(\partial_t+\Delta_{\Phi})u=F(u),\;\;u|_{t=0}=0.
\end{equation}
Assume $F$ to satisfy the conditions $(1)$, $(2)$ and $(3)$ above.  
Then there exists a unique $u^*\in x^\gamma C^{k+2,\alpha}_{\Phi}(M\times[0,T'])$ solution of \eqref{CauchyProbNonLin1} for some $T'$ sufficiently small. 
\end{thm}
\begin{proof}
Let $\eta$ and $T$ be positive numbers to be specified later and set
\begin{equation*}
Z_{\eta,T}:= \left\{u \in x^{\gamma}C^{k+2,\alpha}_{\Phi}(M\times [0,T])\;\big| \; u|_{t=0}=0, \;\|u\|_{k+2,\alpha,\gamma} \le \eta\right\}.
\end{equation*}
$Z_{\eta,T}$ is a closed subset of the Banach space, hence a complete metric space.
For $\mathbf{H}$ as in Theorem \ref{mappropH}, consider the map $\Psi(u) := (\mathbf{H}\circ F)(u)$.
The assumptions on $F$ implies that $\Psi$ maps $x^{\gamma}C^{k+2,\alpha}_{\Phi}(M\times [0,T])$ to itself. 
As mentioned earlier, our aim is to prove $\Psi$ to be a contraction on $Z_{\eta,T}$, for some $\eta$ and $T$ sufficiently small.
Due to linearity of $\mathbf{H}$ we can prove $\Psi_1=\mathbf{H}\circ F_1$ and $\Psi_2=\mathbf{H}\circ F_2$ to be contractions on $Z_{\eta,T}$ instead.
For simplicity let us denote by $\mathbf{C}$ the number
$$\mathbf{C}:=\frac{1}{3\|\mathbf{H}\|_{\op}C_{\eta}}.$$
\medskip

Let us first prove that $\Psi_1$ and $\Psi_2$ map $Z_{\eta,T}$ to itself.
We begin with $\Psi_1$.  
By requiring $T\le \mathbf{C}^2$ one has, for every $u\in Z_{\eta,T}$,
\begin{equation}\label{LastEq1}
\|\Psi_{1}(u)\|_{k+2,\alpha,\gamma} \le \|\mathbf{H}\|_{\op}\sqrt{T}\|F_{1}(u)\|_{k+1,\alpha,\gamma} \le \|\mathbf{H}\|_{\op}\sqrt{T}C_{\eta}\|u\|_{k+2,\alpha,\gamma} \le \frac{\eta}{3}.
\end{equation}
In the above the first estimate follows from the second displayed mapping property of the operator $\mathbf{H}$ in Theorem \ref{mappropH} while the second follows from the assumption on $F_1$.

\medskip
For $\Psi_2$ we argue in a similar manner. 
Let $u$ be an element in $Z_{\eta,T}$.
One has that the chain of inequalities
\begin{equation}\label{LastEq2}
\|\Psi_{2}(u)\|_{k+2,\alpha,\gamma} \le \|\mathbf{H}\|_{\op}\|F_{2}(u)\|_{k,\alpha,\gamma} \le \|\mathbf{H}\|_{\op}C_{\mu}\|u\|^{2}_{k+2,\alpha,\gamma} \le \frac{\eta}{3},
\end{equation}
holds by choosing $\eta\le \mathbf{C}$.
Contrarily to the previous case, here the first estimate follows from the first displayed mapping property of $\mathbf{H}$ in Theorem \ref{mappropH}.

\medskip
It is then clear that, by choosing $\eta\le \mathbf{C}$ and $T\le \mathbf{C}^2$ then \eqref{LastEq1} and \eqref{LastEq2} are both satisfied, resulting in $\Psi=\Psi_1+\Psi_2$ mapping $Z_{\eta,T}$ to itself.
\medskip

The only thing left to prove is $\Psi$ to be a Lipschitz function with Lipschitz constant less than $1$.
For $\eta$ and $T$ as above one sees that, arguing in the exact same way as before:
\begin{align*}
\|\Psi_{1} (u) - \Psi_{1}(u')\|_{k+2,\alpha,\gamma} \le \|\mathbf{H}\|_{\op}\sqrt{T}C_{\eta}\|u-u'\|_{k+2,\alpha,\gamma} \le \dfrac{1}{3}\|u-u'\|_{k+2,\alpha,\gamma}
\end{align*}
and
\begin{align*}
\|\Psi_{2} (u) - \Psi_{2}(u')\|_{k+2,\alpha,\gamma} &\le \|\mathbf{H}\|_{\op}C_{\eta}\max\{\|u\|_{k+2,\alpha,\gamma},\|u'\|_{k+2,\alpha,\gamma}\}\|u-u'\|_{k+2,\alpha,\gamma} \\
&\le \dfrac{1}{3}\|u-u'\|_{k+2,\alpha,\gamma}.
\end{align*}
The above imply, in particular, $\|\Psi(u)-\Psi(u')\|_{k+2,\alpha,\gamma}\le 2/3\|u-u'\|_{k+2,\alpha,\gamma}$. 
Hence $\Psi$ is a Lipschitz function with Lipschitz constant $2/3<1$, thus showing $\Psi$ to be a contraction.
\end{proof}

\bibliographystyle{amsalpha-lmp}

\end{document}